\documentclass[12pt]{amsart}

\usepackage{hyperref}
\usepackage{amsmath}
\usepackage{amsthm}
\usepackage{amsfonts}
\usepackage{amssymb}
\usepackage{bbm}
\usepackage{tikz}
\usepackage{pgf}
\usepackage{graphicx}
\usepackage{ytableau}
\usepackage{todonotes}
\usepackage[margin = 1.5in]{geometry} 
 \numberwithin{equation}{section}
\DeclareMathOperator{\st}{\,|\,}

\DeclareMathOperator{\C}{\mathbb{C}}
\DeclareMathOperator{\R}{\mathbb{R}}
\DeclareMathOperator{\Z}{\mathbb{Z}}

\DeclareMathOperator{\s}{\mathfrak{S}}
\DeclareMathOperator{\D}{\mathcal{D}}
\DeclareMathOperator{\cyc}{cyc}
\DeclareMathOperator{\pw}{pw}

\newcommand{\sg}{\mathfrak S}
\newcommand{\idf}{\mathcal{F}}
\newcommand{\idt}{\mathcal{T}}

\newtheorem{theorem}{Theorem}
\numberwithin{theorem}{section}
\newtheorem{lemma}[theorem]{Lemma}
\newtheorem{proposition}[theorem]{Proposition}
\newtheorem{corollary}[theorem]{Corollary}
\newtheorem{conjecture}[theorem]{Conjecture}

\theoremstyle{definition}

\newtheorem{example}[theorem]{Example}
\newtheorem{remark}[theorem]{Remark}

\begin{document}

\title[Homogenized Linial Arrangement]{The Homogenized Linial Arrangement and Genocchi numbers} 

\author[Lazar]{Alexander Lazar}
\address{Department of Mathematics, University of Miami, Coral Gables, FL 33124}
\email{alazar@math.miami.edu}

\author[Wachs]{Michelle L. Wachs$^{\dag}$}
\address{Department of Mathematics, University of Miami, Coral Gables, FL 33124}
\email{wachs@math.miami.edu}
\thanks{$^{\dag}$Supported in part by NSF Grant
DMS  1502606}

\begin{abstract}  We study the intersection lattice of a hyperplane arrangement recently introduced by Hetyei who showed that the number of regions of the arrangement is a median Genocchi number.  Using a different method, we refine Hetyei's result by providing a combinatorial interpretation of the coefficients  of the characteristic polynomial of the intersection lattice of this arrangement.  We also show that the  M\"obius invariant of the intersection lattice is a (nonmedian) Genocchi number.  The Genocchi numbers  count a class of permutations known as Dummont permutations and the median Genocchi numbers count the derangements in this class.   We show that the signless coefficients of the characteristic polynomial count  Dumont-like permutations with a given number of cycles.
This enables us to derive formulas for the generating function of  the characteristic polynomial, which reduce to known formulas for the generating functions of the Genocchi numbers and the median Genocchi numbers.    As a byproduct of our work, we obtain new models for the Genocchi and median Genocchi numbers.
 \end{abstract}


\date{October 16, 2019}

\maketitle
\tableofcontents

\section{Introduction}

The {\em braid arrangement} (or {\em type A Coxeter arrangement}) is the hyperplane arrangement in $\R^n$, $n \ge 1$, defined by
$$\mathcal A_{n-1} := \{ x_i - x_j = 0 : 1 \le i < j \le n \}.$$  Note that the hyperplanes of $\mathcal A_{n-1}$ divide $\R^n$ into  open cones of the form $$R_\sigma := \{{\bf x} \in \R^n : x_{\sigma(1)} < x_{\sigma(2)}< \dots < x_{\sigma(n)}\},$$
where $\sigma$ is a permutation in the symmetric group $\s_n$. Hence the braid arrangement $\mathcal A_{n-1}$ has  $|\s_n| = n!$ regions.
The regions of the arrangement obtained by intersecting $\mathcal A_2$ with the plane $x+y+z=0$ are shown in the figure below.
  \begin{center}\includegraphics[height=1.3in]{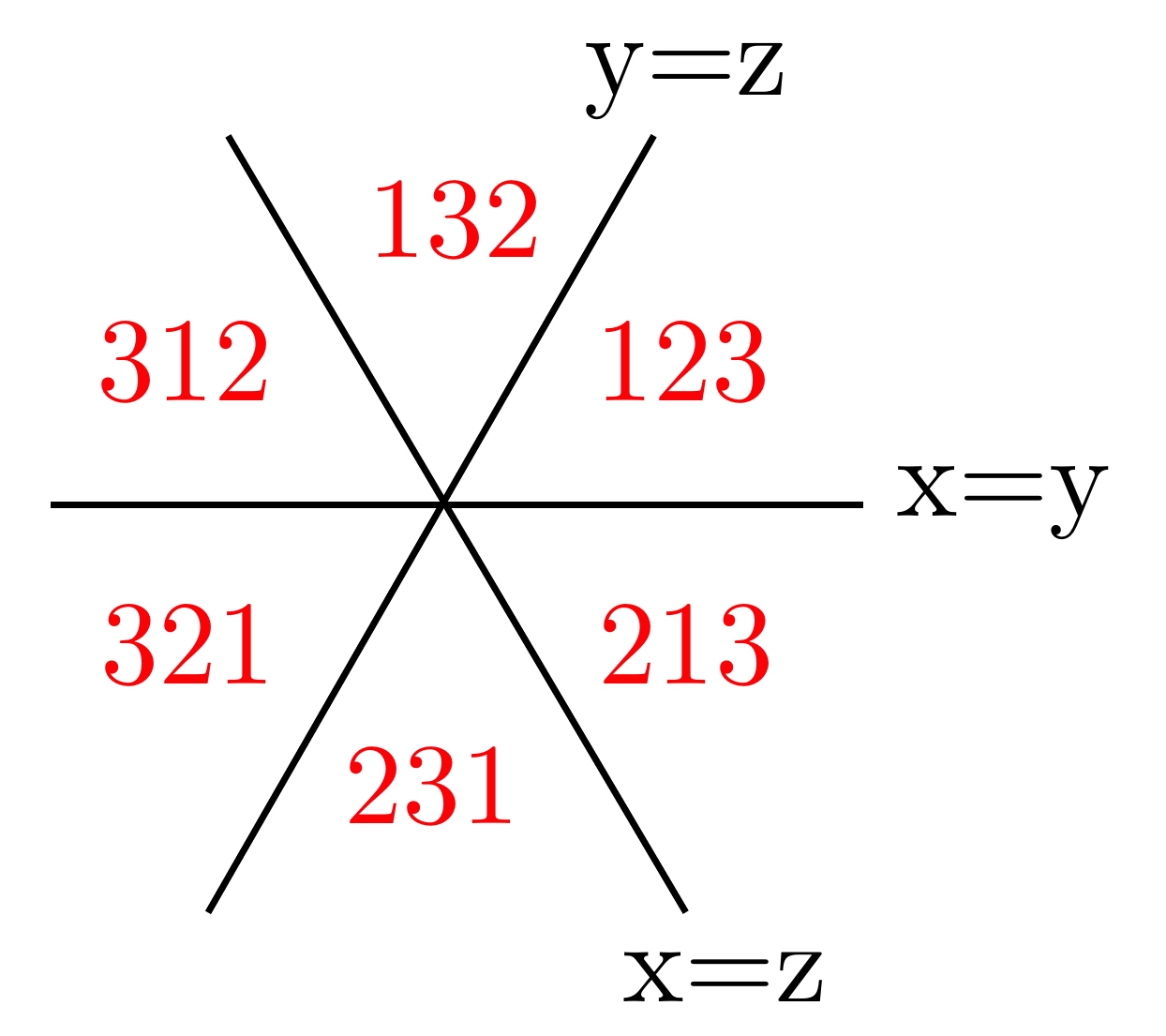}\end{center}

A classical formula of Zaslavsky \cite{Facing_Up_Arrangements} gives the number of regions  of any real hyperplane arrangement  $\mathcal A$ in terms of the M\"obius function of its intersection (semi)lattice $\mathcal L({\mathcal A})$.  Indeed, given any finite, ranked poset $P$ of length $\ell$, with a minimum element $\hat 0$,  the {\it characteristic polynomial} of $P$ is defined to be 
 \begin{equation} \chi_P(t)  :=  \sum_{x \in P}\mu_P(\hat{0},x)t^{\ell-\text{rk}(x)},\end{equation}
  where $\mu_P$ is the M\"obius function of $P$ and $\text{rk}(x)$ is the rank of $x$.  Zaslavsky's formula 
  for the number of regions $r(\mathcal A)$ of $\mathcal A$ is  \begin{equation} \label{zaseq} r(\mathcal{A}) = (-1)^{\ell} \chi_{\mathcal L(\mathcal{A})}(-1). \end{equation}

  It is well known and easy to see that the lattice of intersections of the braid arrangement $\mathcal A_{n-1}$ is isomorphic to  the  lattice $\Pi_n$ of partitions of the set $[n]:=\{1,2\dots,n\}$.  It is also well known that the characteristic polynomial of $\Pi_n$ is given by
\begin{equation} \label{stireq} \chi_{\Pi_n} (t) = \sum_{k=1}^{n} s(n,k) t^{k-1},\end{equation}
  where $s(n,k)$ is the Stirling number of the first kind, which is equal to
  $(-1)^{n-k}$  times the number of permutations in $\s_n$ with exactly $k$ cycles; see \cite[Example 3.10.4]{EC1}. Hence
  $\chi_{\Pi_n}(-1) = (-1)^{n-1} |\s_n|$.  Therefore,  from (\ref{zaseq}), we recover the result observed above that the number of regions of $\mathcal A_{n-1}$ is $n!$.

Deformations of the braid arrangement are often studied in the literature, and in many cases their number of regions is known. Among these deformations are the Linial 
arrangement, Shi arrangement, Catalan arrangement, semiorder arrangement, generic braid arrangement, and semigeneric braid arrangement; see \cite{Stanley_Hyperplanes} 
for a discussion of these arrangements and for a general introduction to the combinatorics of hyperplane arrangements.
The {\em Linial arrangement} is the hyperplane arrangement in $\R^n$ defined by,
$$\{ x_i - x_j = 1 : 1 \le i < j \le n \}.$$
In \cite{defcox} Postnikov and Stanley show that the number of regions of this arrangement is the number of alternating trees on node set $[n+1]$, where a tree is {\em alternating} if each node is either greater than all its neighbors or smaller than all its neighbors.  

Motivated by a problem on enumerating a certain class of tournaments,   Hetyei \cite{Alternation_Acyclic}  introduced the {\em homogenized Linial arrangement}, which 
 is the hyperplane arrangement in 
$$\{ (x_1,\dots, x_n,y_1,\dots,y_{n-1}): x_i \in \R \, \forall i \in [n], \mbox{ and } y_i \in \R \, \forall i \in [n-1]\} = \R^{2n-1}.$$ given by\footnote{Our indexing is justified by the fact  that the length of the intersection lattice of the arrangement is $2n-3$. This is an immediate consequence of Theorem~\ref{bondth}.} $$\mathcal{H}_{2n-3} := \{x_i - x_j = y_i \st 1 \leq i < j \leq n\}  .$$
Note that by intersecting $\mathcal{H}_{2n-3}$ with the subspace $y_1=y_2 =\dots = y_{n-1} =0$, one gets the braid arrangement $\mathcal A_{n-1}$.  Similarly by intersecting  $\mathcal{H}_{2n-3}$ with the subspace $y_1=y_2 =\dots = y_{n-1} =1$, one gets the Linial arrangement in $\R^n$.

In  \cite{Alternation_Acyclic}, Hetyei  proves that 
\begin{equation} \label{heteq} r(\mathcal{H}_{2n-1}) = h_n,\end{equation}
where  $h_n$ is a median Genocchi number,\footnote{In the literature the  median Genocchi number $h_n$ is usually denoted $H_{2n+3}$ or $|H_{2n+3}|$,} 
or equivalently by (\ref{zaseq}) that 
\begin{equation} \label{heteq2} -\chi_{\mathcal L(\mathcal{H}_{2n-1})}(-1) = h_n.\end{equation}
There are numerous characterizations of the median Genocchi numbers in the literature.  One such characterization is given by the following formula for the generating function, which was  obtained by 
Barsky and Dumont  \cite{Barsky_Dumont},
\begin{equation} \label{BDeq} \sum_{n\geq 1}h_{n}x^n = \sum_{n\geq 1}\frac{n!(n+1)!x^{n}}{\prod_{k=1}^n(1+k(k+1)x)}.\end{equation}
A combinatorial characterization in terms of a class of permutations now called Dumont derangements was also given by Barsky and Dumont in \cite{Barsky_Dumont}.   Another was given by  Randrianarivony in   \cite{Randrianarivony_Du_Fo_Poly} 
  in terms of a class of objects called surjective staircases.

Hetyei's proof of (\ref{heteq2}) relies on  the  finite field method of Athanasiadis \cite{Finite_Field_Method}, which is a method for computing  the characteristic polynomial by counting points in the complement of an associated arrangement over a finite field.
Hetyei obtains a recurrence for a refinement of  the associated point count polynomial and relates it to a recurrence of Andrews, Gawronsky, and Littlejohn \cite{Legendre_Stirling} for  Legandre-Stirling numbers.   This yields a formula for  $\chi_{\mathcal L(\mathcal H_{2n-1}) }(-1)$ in terms of the Legandre-Stirling numbers, which is  identical to a formula of Claesson,  Kitaev, Ragnarsson, and Tenner \cite{Boolean_Complex_Ferrers} for the median Genocchi numbers.

In this paper we further study the intersection lattice $\mathcal L(\mathcal H_{2n-1})$ and its
 characteristic polynomial   $\chi_{\mathcal L(\mathcal H_{2n-1}) }(t)$ using an approach quite different from Hetyei's.  The  first few characteristic polynomials  and their values at $t=1$ and $t=0$ are given in the table below. 
 
 \vspace{.1in}
 \small{ \begin{center}\begin{tabular}{c|c|c|c}
  $n$ & $\chi_{\mathcal L(\mathcal{H}_{2n-1})}(t)$  & $t=-1$ &  $t=0$  \\[-2\medskipamount] & &  \\  \hline  & & \\[-2\medskipamount]      
$1$ & $t-1$ & $-2$ & $-1$
\\   
$2$ &$ t^3 - 3t^2 + 3t - 1$   & $-8$ & $-1$
\\   
$3$ &  $t^5 - 6t^4 + 15t^3 - 19t^2 + 12t - 3$  & $-56$ & $-3$
\\ 
$4$ & $t^7 -10t^6 + 45t^5 - 115t^4 + 177t^3 -162t^2 + 81 t  - 17$ & $-608$ & $-17$

\end{tabular} \end{center}}

\vspace{.1in}

 We  begin by showing that the intersection lattice $ \mathcal L(\mathcal H_{2n-1}) $ is isomorphic to the bond lattice of a certain bipartite graph, which belongs to the class of Ferrers graphs introduced by Ehrenborg and van  Willigenburg \cite{Enumerative_Ferrers_graphs}.  This enables us to 
refine Hetyei's result by 
deriving a combinatorial formula for the M\"obius invariant of the  lower intervals of  $\mathcal L(\mathcal H_{2n-1})$  
   in terms of a class of permutations that we call D-permutations, which are similar to the Dumont permutations.
   We obtain the following analog of (\ref{stireq}):
\begin{equation} \label{intorchardumeq} \chi_{\mathcal L(\mathcal H_{2n-1})}(t) = \sum_{k=1}^{2n} s_D(2n,k) t^{k-1},\end{equation}
where $(-1)^{k}s_D(2n,k)$ is equal to  the number of D-permutations on $[2n]$ with exactly $k$ cycles.

 The D-permutations have a simple description.  A permutation $\sigma$ on  $[2n]$ is a {\em D-permutation} if $ i \le \sigma(i) $ whenever $i$ is odd, and $ i \ge \sigma(i) $ whenever $i$ is even.   A consequence of (\ref{intorchardumeq}) and (\ref{heteq2}) is  that the median Genocchi number $h_n$ is equal to the number of D-permutations on $[2n]$.    A particularly interesting feature of this permutation model  for the median Genocchi numbers is that the (nonmedian) Genocchi number\footnote{In the literature the  Genocchi number $g_n$ is usually denoted $G_{2n}$ or $|G_{2n}|$.}  $g_n$ enumerates a subset of the set of  D-permutations on $[2n]$; namely the set of D-cycles on $[2n]$.      This follows from a simple bijection with a  set of  permutations on $[2n-1]$ shown by  Dumont in \cite{Interpretations_Combinatoires} to have cardinality equal to  $g_n$.  This feature is in contrast with the Dumont permutation model  in which the Genocchi numbers enumerate the full set of Dumont permutations on $[2n]$ and the median Genocchi numbers enumerate a subset of the these.
  
Our formula (\ref{intorchardumeq}) therefore implies 
 \begin{equation} \label{intromueq} \mu_{\mathcal L(\mathcal{H}_{2n-1})}(\hat 0, \hat 1) = \chi_{\mathcal L(\mathcal{H}_{2n-1})}(0) = -g_n,\end{equation} where  and 
 $\hat 0$ and $\hat 1$ are the minimum and maximum elements of $\mathcal L (\mathcal H_{2n-1})$, respectively.  Hence the Genocchi numbers, as well as the median Genocchi numbers, play a fundamental role in the study of the homogenized Linial arrangement.

 We recover Hetyei's result (\ref{heteq2}),
 by constructing a bijection from  the full set of D-permutations on $[2n]$ to  a set of surjective staircases  shown by  Randrianarivony  in  \cite{Randrianarivony_Du_Fo_Poly}   to have cardinality equal to the median Genocchi number $h_n$.
Moreover, this bijection  and the theory of surjective staircases enable us to prove two generating function formulas for the characteristic polynomials. One such formula is
\begin{equation}\label{introgenchareq} \sum_{n\geq 1} \chi_{\mathcal L(\mathcal H_{2n-1})}(t) \, x^n =  \sum_{n\geq 1}\frac{  (t-1)_{n-1} (t-1)_{n} \,x^n}{\prod_{k=1}^n(1-k(t-k)x)},\end{equation}
where $(a)_{n}$ denotes the falling factorial  $a(a-1)\cdots (a-n+1)$.  Note that (\ref{introgenchareq})
 reduces to the  Barsky-Dumont formula (\ref{BDeq}) for $\sum_{n\ge 1} h_n x^n$ when $t$ is set equal to $-1$, and to a similar formula of Barsky and Dumont in \cite{Barsky_Dumont} for  $\sum_{n\ge 1} g_n x^n$ when $t$ is set equal to $0$.  
 
 In addition to the D-permutation model,   our study  of the characteristic polynomial leads to other 
interesting combinatorial  models for the median Genocchi numbers and  the Genocchi numbers.  For instance, using our bond lattice result and Hetyei's formula (\ref{heteq2}), we show that a formula of Chung and 
Graham \cite{cover_polynomial} for  chromatic polynomials of  incomparability graphs
 yields another nice permutation model 
 for the median Genocchi numbers: $h_n$ is
 the number of
permutations $\sigma$ on $[2n]$ such that $i > \sigma(i)$ only if $i$ is even and $\sigma(i) $ is odd.  This leads to the conjectural interpretation of 
 $-\chi_{\mathcal L(\mathcal H_{2n-1})}(-t) $ as the enumerator of such permutations by their number of cycles, which would imply that the Genocchi number
$g_n$ is equal to the number of such permutations that are cycles.

The paper is organized as follows.  In Section~\ref{prelimsec} we review some basic material on hyperplane arrangements, bond lattices, and  Genocchi numbers.  Our result that the intersection lattice of the homogenized Linial arrangement is isomorphic to a bond lattice  is proved in Section~\ref{bondsec}.    This bond lattice  has  a nice description as  the induced subposet   of the partition lattice $\Pi_{2n}$ consisting of partitions all of whose nonsingleton blocks have odd minimum and even maximum.   Using Whitney's NBC theorem, we show that the signless M\"obius invariant of each lower interval $[\hat 0, \pi]$ of the bond lattice  enumerates a certain class of alternating forests on $[2n]$.

In Section~\ref{Dsec},  we prove (\ref{intorchardumeq}) and a more general formula for  chromatic polynomials of  general Ferrers graphs by constructing a bijection between the  alternating forests of Section~\ref{bondsec} and the D-permutations on $[2n]$.  We also give some alternative formulas for the characteristic polynomials in terms of D-permutations and we present some of their consequences.  We observe that the characteristic polynomial $\chi_{\mathcal L(\mathcal H_{2n-1})}(t)$ is divisible by $(t-1)^3$ and we give  combinatorial and geometric interpretations  of the polynomial $(t-1)^{-3}\chi_{\mathcal L(\mathcal{H}_{2n-1})}(t)$.

Section~\ref{DtoSsec} contains the bijection from the D-permutations on $[2n]$ to the set of surjective staircases whose cardinality is equal to $h_n$.  
This bijection, the  formulas of Section~\ref{Dsec} for the characteristic polynomial, and the theory of surjective staircases are used  to prove  (\ref{introgenchareq}) and a similar generating function formula.

In Section~\ref{othersec}, we present the   permutation models for the Genocchi numbers and median Genocchi numbers that arise  from the Chung-Graham result on chromatic polynomials.
Further work on a type B analog and a Dowling arrangement generalization, which will appear in a forthcoming paper, is discussed in 
Section~\ref{dowsec}.

Most of the results of this paper and of the forthcoming paper on the Dowling generalization were announced in the extended abstract \cite{Extended_Abstract}.

\section{Preliminaries} \label{prelimsec}

\subsection{Hyperplane Arrangements}

Let $k$ be a field (here $k$ is $\R$ or $\C$). A {\em hyperplane arrangement} $\mathcal{A} \subseteq k^n$ is a finite collection of affine codimension-$1$ subspaces of $k^n$.
The {\em intersection poset} of $\mathcal{A}$ is the poset $\mathcal{L}(\mathcal{A})$ of intersections of hyperplanes in $\mathcal{A}$ (viewed as affine subspaces of $k^n$), partially-ordered by reverse inclusion. If 
$\bigcap_{H \in \mathcal A} H \ne \emptyset$ then the intersection poset is a geometric lattice, otherwise it's a geometric semilattice.

If $\mathcal{A}$ is a real hyperplane arrangement in $\R^n$, then its complement $\R^n\setminus \mathcal{A}$ is disconnected. By the number of regions $r(\mathcal A)$  of $\mathcal A$ we mean the number of  connected components of $\R^n\setminus \mathcal{A}$. 
This number can be detected solely from $\mathcal{L}(A)$ as Zaslavsky's formula (\ref{zaseq}) shows.

\subsection{The Bond Lattice of a Graph} \label{bondsubsec}

Let $G=(V,E)$ be a graph.  Given a subset $B$ of $V$, let $G|_B$ denote the induced subgraph of $G$ with vertex set $B$.  Let $\Pi_V$ denote the lattice of partitions of the set $V$ ordered in the usual way  by reverse refinement. If $V=[n]:=\{1,2,\dots,n\}$ we write $\Pi_n$ for $\Pi_V$.  The {\em bond lattice of $G$} is the induced subposet $\Pi_G$ of  $\Pi_V$ consisting of partitions $\pi = B_1|\cdots|B_k$ such that $G|_{B_i}$ is connected for all $i$.  It is  well known that the chromatic polynomial ${\rm ch}_G(t)$ of $G$ satisfies \begin{equation}  \label{chromchareq} {\rm ch}_G(t) = t \chi_{\Pi_G}(t). \end{equation}

If $V=[n]$, we can also associate to $G$ the \emph{graphic hyperplane arrangement}
$$\mathcal{A}_G = \{x_i - x_j = 0 \st \{i,j\} \in E(G)\} \subseteq \R^n.$$
It is well known   and easy to see that 
\begin{equation} \label{bondeq} \Pi_G \cong \mathcal{L}(\mathcal{A}_G) . \end{equation}

Note that $\Pi_n$ is the bond lattice of the complete graph $K_n$ and that the braid arrangement $\mathcal A_{n-1}$ is the graphical arrangement associated with $K_n$.  Another example is given below.

\begin{center}
\includegraphics{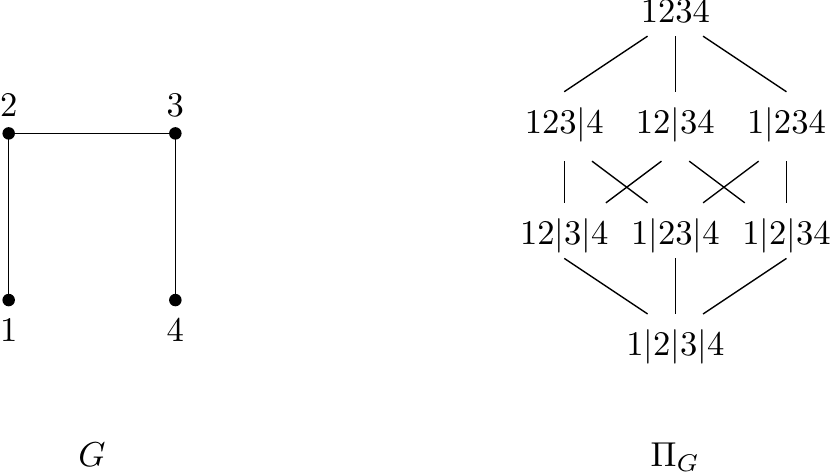}
\end{center}

Broken circuits provide a useful means of computing the M\"obius function of the bond lattice of a graph $G=(V,E)$ (or more generally, of a geometric lattice).  
Fix a total ordering of $E$ and let $S$ be a subset of $E$. Then $S $  is called a {\em broken circuit} if it is the edge set  of a cycle of $G$  with its smallest edge (with respect to this ordering) removed. If $S$ does not contain a broken circuit, we say that $S$ is a {\em non-broken circuit set} or {\em NBC} set.
Clearly the edge set of a cycle always contains a broken circuit. Hence if $S$ is an NBC set,  $(V,S)$ is a forest.  We call  $(V,S)$ for which $S$ is an NBC set an {\em NBC forest} of $G$.

Given any  $S\subseteq E$, let $\pi_S$ be the partition of $V$ whose blocks are the vertex sets of the connected components of the graph $(V,S)$. 
The following formula is due to Whitney
 \cite[Section 7]{Logical_Math}; a generalization for geometric lattices is due to Rota \cite[Pg. 359]{Found_Comb_1}. 
For $\pi \in \Pi_G$,
\begin{equation} \label{RWeq} (-1)^{\text{rk}(\pi)}\mu(\hat{0},\pi) = \#\{\mbox{NBC forests $(V,S)$  of } G :  \pi_S = \pi\}.\end{equation}

Given a  tree $T$ whose node set is a subset of $\Z_{> 0}$, we say the tree is {\em increasing} if when $T$ is rooted at its smallest node, each nonroot vertex  is larger than its parent. A  forest on a subset of $\Z_{> 0}$ is said to be {\em increasing} if it consists of increasing  trees.  Note  that if $G$ is $K_n$ then by ordering the edges lexicographically with the smallest element as the first component, the NBC forests of $G$ are exactly  the increasing forests on $[n]$.  It therefore follows from (\ref{RWeq}) and (\ref{zaseq}) that $r(\mathcal A_{n-1}) $ equals the number of increasing forests on $[n]$, which can easily be shown to be equal to $n!$.

For any tree $T$ (rooted or unrooted), let $|T|$ denote the number of nodes of $T$. 
Recall that a \emph{plane tree} is a rooted tree in which the children of each node are linearly ordered. Let $T$ be a plane tree on a subset of $\Z_{>0}$. When $T$ is drawn in the plane, the children of each vertex $v$ are drawn from left to right according to their linear order.

Now let $T$ be a plane   tree  on a subset of $\Z_{> 0}$.  When drawn in the plane the order is depicted from left to right.  For each vertex $v$, the order of its children determines  a  ``left to right" order of the subtrees rooted at the children.  If  $|T|>1$, let  $T_1,T_2,\dots,T_k$ be the subtrees rooted at the children of the root $r$ of $T$ ordered from ``left to right".  The   {\em postorder word} $\pw(T)$ of $T$ is defined recursively as the concatenation  $$\pw(T) = \pw(T_1) \cdot \pw(T_2)\cdot \,\, \cdots \,\,  \cdot \pw(T_k)\cdot  r,$$
if $|T| > 1$ and by $\pw(T)= r$ if $|T| = 1$.

 Every increasing tree $T$ on node set $V\subset \Z_{> 0}$  can be viewed as a plane  tree on $V$ by rooting it at its smallest node and  then ordering the children of each node in increasing order.  It is not difficult to prove that the map that sends an increasing tree $T$ on $V$ to the permutation  whose cycle form is $(\pw(T))$, is a bijection from the  set of increasing trees to the set of cycles in $\sg_V$. This bijection extends in the obvious way to a bijection from the set of increasing forests on  $V$ to the set of permutations in $\sg_V$ whose cycles correspond to the trees of the forest.  Thus this bijection and (\ref{RWeq}) yield another way of proving  (\ref{stireq}).

 \subsection{Genocchi and median Genocchi Numbers} \label{gensec}

The Genocchi numbers and median Genocchi numbers are classical sequences of numbers, which have been defined in many ways.
Below we take a characterization of Dumont   \cite[Section 6]{Interpretations_Combinatoires} as our definition of the Genocchi numbers and  a characterization of Barsky and Dumont \cite{Barsky_Dumont} as our definition of the median Genocchi numbers.

A {\em Dumont permutation} is a permutation $\sigma \in \mathfrak{S}_{2n}$ such that $2i>\sigma(2i) $ and $2i-1 \le \sigma(2i-1)$ for all $i=1,\dots,n$. A {\em Dumont derangement} is a Dumont permutation without fixed points, i.e.,  $2i > \sigma(2i) $ and $2i-1 < \sigma(2i-1)$ for all $i=1,\dots,n$.

\begin{example}
When $n=2$, the Dumont permutations on $[4]$ (in cycle form) are $$(1,2)(3,4) \qquad (1,3,4,2) \qquad (1,4,2) (3).$$ 
When $n=3$, the Dumont derangements on $[6]$ are:
\begin{center}$\begin{array}{cccc}
(1,3,5,6,4,2) & (1,3,4,2)(5,6) & (1,2)(3,4)(5,6) & (1,2)(3,5,6,4)\\
(1,4,3,5,6,2) & (1,5,6,3,4,2) & (1,5,6,2)(3,4) & (1,4,2)(3,5,6) .
\end{array}$\end{center}
\end{example}

For  $n \ge 1$, the (signless) {\em Genocchi number} $g_n$ is defined to be the number of Dumont permutations on $[2n-2]$,  and for $n \ge 0$, the (signless) {\em median Genocchi number} $h_n$ is defined to be the number of Dumont derangements on $[2n+2]$.  We list the Genocchi numbers and the median Genocchi numbers for small values of $n$ in the table below.

\begin{center}\begin{tabular}{|c|c|c|c|c|c|c|c|}
\hline{\color{red} $n$}  & {\color{red}0} & {\color{red}1}& {\color{red}2}& {\color{red}3} & {\color{red}4}& {\color{red}5} & {\color{red}6}
\\ \hline\hline 
$g_n$ & &$1$ & $1$ & $3$ & $17$ & $155$ & $2073$ 
\\   \hline 
$h_n$ &  $1$ & $2$& $8$ & $56$ & $608$ & $9440$ & $198272$
\\ \hline 
\end{tabular} \end{center}
  Our notation is nonstandard in that  $g_n$  is usually denoted $G_{2n}$ or $|G_{2n}|$, while $h_n$ is usually denoted $H_{2n+3}$ or $|H_{2n+3}|$.
  The Genocchi numbers and median Genocchi numbers are also known as Genocchi numbers of the first and second kind, respectively. 

\vspace{.1in} Another permutation characterization of the Genocchi numbers obtained by Dumont in \cite[Section 6]{Interpretations_Combinatoires} is given by
\begin{equation} \label{otherDumontEq} g_n = |\{\sigma \in \sg_{2n-1} : \forall \,i\in [2n-2], \,\sigma(i) > \sigma(i+1) \mbox{ if and only if $\sigma(i)$ is even} \}|.
\end{equation}

We also mention the following  exponential generating function formula for the Genocchi numbers (see \cite[page 305]{Interpretations_Combinatoires})
$$\sum_{n \ge 1} g_{n} \,\, \frac{x^{2n}}{(2n)!} = x \tan \frac x 2 .$$  
In \cite{Barsky_Dumont}, two formulas for the generating function of the Genocchi numbers and two formulas for the generating function of the median Genocchi numbers are given. We list these four generating function formulas here. 

\begin{eqnarray}  \label{BD2eq}
 \sum_{n\geq 1} g_{n}z^n &=& \sum_{n\geq 1}\frac{(n-1)!\,n! \, z^{n}}{\prod_{k=1}^n(1+k^2z)} 
  \\   \label{BD1eq}\sum_{n\geq 0}h_{n}z^n &=& \sum_{n\geq 0}\frac{n!\,(n+1)!\,z^{n}}{\prod_{k=1}^n(1+k(k+1)z)}
  \\   \label{BD4eq}\sum_{n\geq 0}g_{n+1}z^n &=& \sum_{n\geq 0}\frac{(n!)^2\,z^{n}}{\prod_{k=1}^n(1+k^2z)}
 \\  \label{BD3eq} \sum_{n\geq 1}h_{n-1}z^{n} &=& \sum_{n\geq 1}\frac{(n!) ^2z^n}{\prod_{k=1}^{n}(1+k(k+1)z)}.
\end{eqnarray}
 Equation (\ref{BD4eq}) is due to Carlitz \cite{Carlitz_genocchi} and to Riordan and Stein \cite{Riordan_Stein_genocchi}, and the rest are due to Barsky and Dumont \cite{Barsky_Dumont} .

\section{The intersection lattice is a bond lattice} \label{bondsec}

In this section we show that $\mathcal L(\mathcal H_{2n-1})$ is isomorphic to  the bond lattice of a  bipartite graph. This enables us to give a characterization of the intersection lattice $\mathcal L(\mathcal H_{2n-1})$ as an induced subposet of $\Pi_{2n}$ and to compute its M\"obius function by counting NBC sets.  

 Let $\Gamma_{2n}$ be the bipartite graph on $\{1,3,5,\dots,2n-1\} \sqcup \{2,4,6,\dots, 2n\}$ with an edge $\{2i-1,2j\}$ whenever $ 1 \le i\leq j \le n$.  The graph $\Gamma_6$ is shown below.

\vspace{.1in}

\begin{center}
\includegraphics{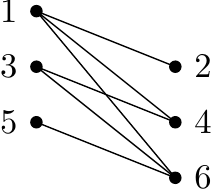}
\end{center}

\vspace{.1in}

More generally for any finite subset $V$ of $\Z_{> 0}$, let $\Gamma_V$ be the bipartite graph on $V$ with edge set $$E=\{\{u,v\} : u<v \in V, \mbox{ $u$ is odd and $v$ is even} \}.$$

\begin{remark} \label{ferrem} The graphs $\Gamma_{V}$ are easily seen to belong to the class of \emph{Ferrers graphs}, which were introduced by Ehrenborg and van Willegenburg in \cite{Enumerative_Ferrers_graphs}.   In fact, it is not much more difficult to show 
that all Ferrers graphs are  of this form.  Ferrers graphs have been further studied in (among others) \cite{Boolean_Complex_Ferrers} and \cite{Selig_Smith_Steingrimsson}. \end{remark}

\begin{theorem} \label{bondth} The poset isomorphism $\mathcal{L}(\mathcal{H}_{2n-1}) \cong \Pi_{\Gamma_{2n}}$ holds for all $n \ge 1$. \end{theorem}
\begin{proof} Let $(e_1,e_2, \dots,e_{2n+1}) $ be the standard basis for $\R^{2n+1}$. For $1\le i \le j \le n$, let
$$ H_{i,j} = \{v \in \R^{2n+1} : (e_{i}-e_{n+1+i} - e_{j+1}) \cdot v = 0 \} $$
and 
$$K_{i,j} = \{v \in \R^{2n+1} : (e_{2i-1}-e_{2j}) \cdot v = 0 \} .$$
Clearly $\{H_{i,j} : 1\le i \le j \le n \}$ is precisely the arrangement $\mathcal H_{2n-1}$. Note that 
$\mathcal K_{2n}:=\{ K_{i,j} : 1\le i \le j \le n \}$ is the graphical arrangement $\mathcal A_{G}$, where $G$ is the graph on $[2n+1]$ whose edge set is equal to $E(\Gamma_{2n})$.  In other words, $G$ is $\Gamma_{2n}$ with an appended  isolated node $2n+1$.  Since adding an  isolated node to a graph yields an isomorphic bond lattice,
 by (\ref{bondeq}) we have $$\mathcal L(\mathcal K_{2n}) \cong \Pi_{\Gamma_{2n}}.$$  

 We will prove the result by producing a vector space isomorphism $\psi: \R^{2n+1} \to \R^{2n+1}$ that takes $\mathcal K_{2n}$ to $\mathcal{H}_{2n-1}$.  Indeed,  such a map will induce an  isomorphism from $\mathcal{L}(\mathcal{K}_{2n})$ to $\mathcal{L}(\mathcal{H}_{2n-1}) $.  
  
 First consider the linear operator $\phi:\R^{2n+1} \to \R^{2n+1}$ defined   on the standard basis by letting 
 $\phi(e_{2i-1}) = e_i - e_{n+1+i}$ and $\phi(e_{2i}) = e_{i+1}$ for all $i \in [n]$ and letting $\phi(e_{2n+1}) = e_{2n+1}$ .  Let $A$ be the matrix of $\phi$ with respect to the standard basis.  One can easily check that $|\det A | = 1$.    Now let $\psi :\R^{2n+1} \to \R^{2n+1}$ be the linear operator whose matrix with respect to the standard basis is the transpose of $A^{-1}$.  Clearly $\psi$ is an isomorphism.  
 
 We claim that $\psi$ takes the hyperplane $K_{i,j}$ to the hyperplane $H_{i,j}$ for all $1\le i \le j\le n$. To prove the claim, let $v \in  K_{i,j}$, so $(e_{2i-1}-e_{2j}) \cdot v = 0$.   We have $$\phi (e_{2i-1}-e_{2j}) \cdot \psi(v) = (e_i - e_{n+1+i} - e_{j+1}) \cdot \psi(v).$$  
 We also have \begin{align*} \phi (e_{2i-1}-e_{2j}) \cdot \psi(v) &=(A (e_{2i-1}-e_{2j})) \cdot ((A^{-1})^T v) \\ &= (A^{-1}A  (e_{2i-1}-e_{2j})) \cdot  v \\ &=  (e_{2i-1}-e_{2j}) \cdot  v \\  &= 0 .\end{align*}
 It follows that $(e_i - e_{n+1+i} - e_{j+1}) \cdot \psi(v) = 0$.  Hence $\psi(v )$ is in $H_{i,j}$, which proves the claim.  It follows from the claim that $\psi$ takes $\mathcal K_{2n}$ to $\mathcal{H}_{2n-1}$ as desired.
\end{proof}

Hence, to study $\mathcal{L}(\mathcal{H}_{2n-1})$, it suffices to study $\Pi_{\Gamma_{2n}}$. 

\begin{proposition}
For all finite subsets $V $ of $\Z_{> 0}$, the bond lattice $\Pi_{\Gamma_{V}}$ is  the induced subposet of $\Pi_{V}$ consisting of those partitions $\pi$ for which each nonsingleton block of $\pi$ has an odd minimum and an even maximum.
\end{proposition}

\begin{proof}
Suppose that  $\pi  \in \Pi_{\Gamma_{V}}$. Let $B$ be a nonsingleton block of $\pi$ and let $u= \min B$. 
Since
 $\Gamma_{V}|_{B}$ is connected, there must be an element $v \in B$  such that $\{u,v\} $ is an edge of $\Gamma_{V}$.  By definition of $\Gamma_{V}$, since $u <v$, the node $u$ must be odd.  A similar argument tells us that $\max B$ is even.

Conversely, suppose that we have a partition $\pi = B_1|\cdots|B_k \in \Pi_{V}$ such that each nonsingleton $B_i$ has an odd minimum and an even maximum. We wish to show that $\Gamma_{V}|_{B_i}$ is connected. Indeed, let $M_i = \max B_i$ and $m_i = \min B_i$. By the definition of $\Gamma_{V}$, each odd vertex in $B_i$ shares an edge with $M_i$ and each even vertex of $B_i$ shares an edge with $m_i$. Since, in particular, this means that $m_i$ and $M_i$ share an edge, $\Gamma_{V}|_{B_i}$ is connected. Hence $\pi  \in \Pi_{\Gamma_{V}}$. \end{proof}

We  will use the Rota-Whitney formula (\ref{RWeq})  to evaluate the M\"obius funtion on each lower interval $[\hat 0 ,\pi]$ of $\Pi_{\Gamma_{V}}$.   Let $T$ be a tree whose nodes are in $ \Z_{> 0}$.  We say that $T$ is {\em increasing-decreasing (ID)} if, when $T$ is rooted at its largest  node, each internal node $v$ satisfies
\begin{enumerate}
\item if $v$ is odd then $v$ is less than all its descendants and all its children are even,
\item if $v$ is even then $v$ is greater than all its descendants and all its children are odd. \end{enumerate}

An an \emph{increasing-decreasing (or ID) forest}  is a forest, each of whose components are increasing-decreasing trees.
An example of an ID forest is given in the figure below.

\ \begin{center} \includegraphics[height=1.3in]{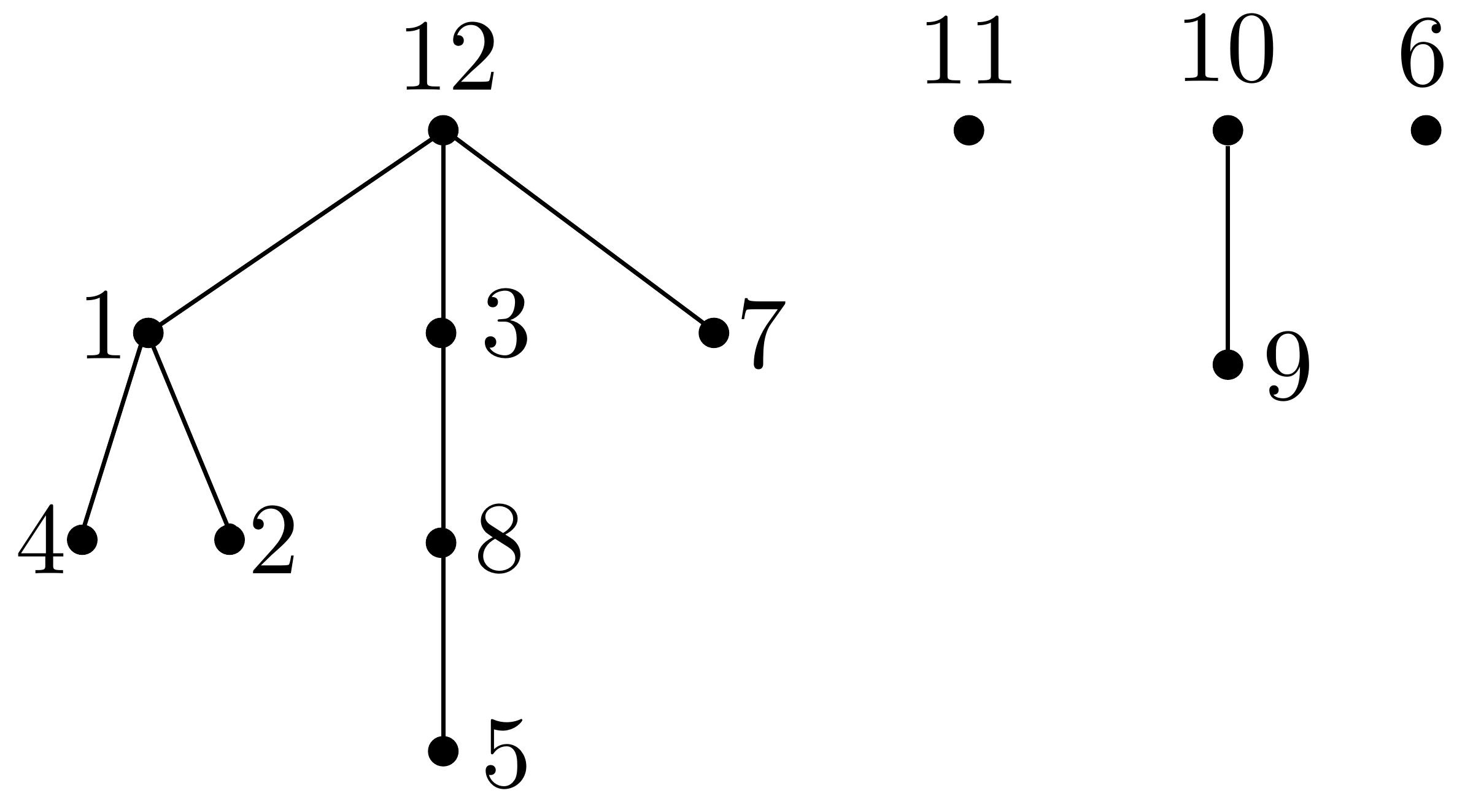} \end{center}

\begin{proposition} \label{larsmlprop} Let $T$ be a tree whose nodes are in $\Z_{> 0}$.  Then $T$ is an ID tree if and only if its internal nodes satisfy conditions (1) and (2) of the definition of ID  tree when $T$ is rooted at its smallest node.
\end{proposition}

\begin{proof} Assume $|T| > 1$. Let $T$ be an ID tree with odd smallest node $s$ and even largest node $t$.  Root $T$ at $t$.  Let $T_1,\dots, T_k$ be the subtrees rooted at  the children of $t$.  Clearly, $s$ must be the root of one of these subtrees say $T_1$.   Let $S_1,\dots,S_j$ be the subtrees rooted at the children of $s$. Let $T^\prime$ be the tree obtained by removing $T_1$ from $T$. 
  When we root $T$ at $s$ we get the rooted tree in which the subtrees rooted at the children of $s$ are  $S_1,\dots, S_j, T^\prime$.  
Since $T$ is an ID tree,  these subtrees  $S_1,\dots, S_j, T^\prime$ must be ID trees rooted at their largest nodes, which are even.  Hence the internal nodes of $T$ rooted at $s$  satisfy (1) and (2) of the definition of ID tree.

The proof of the converse is analogous with the roles of even and odd exchanged, and the roles of larger and smaller exchanged.
\end{proof}

The following result is implicit in  the proof of Theorem 24 of \cite{Selig_Smith_Steingrimsson}.  We include a proof for the sake of completeness.
\begin{theorem}\label{doubth} For all $\pi \in \Pi_{\Gamma_{V}}$, we have that $(-1)^{|\pi|} \mu(\hat 0,\pi)$ equals the number of ID forests on $V$ whose trees have nodes sets equal to the blocks of $\pi$. Consequently,
$  -\mu_ {\Pi_{\Gamma_{V}}}(\hat 0, \hat 1) $ is equal to the number of ID trees on $V$ and 
 $$ \chi_{\Pi_{\Gamma_{V}}}(t) = \sum_{k=1}^{2n} (-1)^{k} |\mathcal F_{V,k}| t^{k-1},
$$
 where $\mathcal F_{V,k}$ is the set of ID forests on $V$ with exactly $k$ trees.
\end{theorem}

\begin{proof}
We prove this result by showing that under an appropriate ordering of the edges of $\Gamma_{V}$, the NBC forests of $\Gamma_{V}$ are  precisely the ID forests on $V$.
The total order on the edges of $\Gamma_{V}$ we use is as follows: First order the odd vertices in {\em increasing} order and the even vertices in \emph{decreasing} order, and then fix any linear extension $\prec$ of the product order of the two linear orders.   By associating the edge $\{2i-1,2j\} $ with the  ordered pair $(2i-1,2j)$, we get a total  order on the edges of $\Gamma_{V}$.
For instance, we have $\{1,8\} \prec \{1,6\} \prec \{3,2\}$.

  Since adding an edge between two different components of a forest can never complete a cycle, a forest is an NBC forest of $\Gamma_{V}$ if and only if  its trees are NBCs trees of  the subgraph  of $\Gamma_{V}$ induced by the  node set   of the tree.  
By definition, a forest is an ID forest if and only if  its trees are ID trees.  Hence, it suffices to prove that   the NBC trees of $\Gamma_{V}$  are precisely the ID trees on $V$.  We can assume $|V| > 1$, $\min V$ is odd and $\max V$ is even.

Let  $T$ be an ID tree on $V$. Suppose that $e = \{2i-1,2j\}$ is an edge of $\Gamma_{V}$ that is not in $T$. If we add $e$ to $T$, we create a unique cycle $\rho$, which  contains $e$.   To show that $T$ is NBC, it suffices to find an edge of $\rho$  that strictly precedes $e$ in the order $\prec$.  Root $T$ at its largest node and let $h$ be the youngest common ancestor of $2i-1$ and $2j$ (that is, the furthest from the root).  Clearly $h$ is a node of $\rho$. 

Suppose $h$ is odd. Then $h \leq 2i-1$ since $T$ is an ID tree.  If $h$ is the parent of $2j$ then  $h$ cannot be $2i-1$ since $\{2i-1,2j\}$ is not an edge of $T$. Therefore $h < 2i-1$, which implies that  $\{h,2j\} \prec \{2i-1,2j\} = e$.  Thus $\{h,2j\}$ is the edge of $\rho$ that we seek.
If $h$ is not the parent of $2j$ then let $2k$ be the child of $h$ that is a proper ancestor of $2j$.  We thus have $2k > 2j$, which implies $\{h, 2k\} \prec \{2i-1, 2j\} = e$.  Thus in this case, $\{h, 2k\}$ is the edge of $\rho$ that we seek.

The case that $h$ is even is handled analogously with the role of even and odd exchanged and the role of $<$ and $>$ exchanged.
We can now conclude that $T$ is an NBC tree of $\Gamma_V$.

Conversely, let $T$ be an NBC tree of $\Gamma_V$.  Suppose that $T$ is not an $ID$ tree.  Then $T$  rooted  at its largest node  has a node $x$ with a grandchild $y$ satisfying  $$x>y \mbox{  if $x$ is odd and } x <y \mbox{ if $x$ is even}.$$  Choose $x$ so that its distance in $T$ from the root is minimal. We claim that $x$ cannot be the root of $T$.  Indeed, if $x$ is even then it is not   the largest node of $T$.  Thus $x$ is not the root.   If $x$ is odd then it is not the root since the largest node of $T$ must be even.  Let $u$ be the parent of $x$ and let $v$ be the parent of $y$, which implies that $x$ is that parent of $v$.   Thus,  $u,x,v,y$ forms a path in $T$.    

We claim that $\{u,y\}$ is an edge of $\Gamma_V$.  Indeed, if $u$ is even then $x$ is odd.  We thus have  $u > x >y$.  Also  $y$ is odd since it is a grandchild of $x$.  It follows that $\{u,y\}$ is an edge of $\Gamma_V$.  The case  when $u$ is odd is handled analogously.  

Since $\{u,y\}$ is an edge of $\Gamma_V$ we can use it to complete the cycle induced by the node set $\{u,x,v,y\}$.  Clearly $\{u,y\} \prec \{u,x\}$. By the minimality of the choice of $x$, since $u$ is closer to the root than $x$ is, we have $u < v$ if $u$ is odd and $u >v$ if $u$ is even.
Hence $\{u,y\} \prec \{u,x\}, \{v,x\}, \{v,y\}$.  This implies that the subgraph of $T$ induced by $\{u,x,v,y\}$ is a broken circuit which contradicts the hypothesis that $T$ is an NBC tree.  Therefore $T$ is indeed an ID tree. \end{proof}

By Theorem~\ref{bondth}, we have the following consequence of Theorem~\ref{doubth}.
\begin{corollary}\label{doubcor} For all $n \ge 1, $   
$$ \chi_{\mathcal L(H_{2n-1})}(t) = \sum_{k=1}^{2n} (-1)^{k} |\mathcal F_{2n,k}| t^{k-1},
$$
 where $\mathcal F_{2n,k}$ is the set of ID forests on $[2n]$ with exactly $k$ trees.
\end{corollary}

Using Hetyei's formula (\ref{heteq2}) and Corollary~\ref{doubcor},   we obtain the following combinatorial interpretation  of the median Genocchi numbers. In Section~\ref{DtoSsec}, we give an alternative proof that does not rely on these results.
\begin{corollary}For all $n\ge 1$, the  number of ID forests on $[2n]$ is equal to $h_n$.
\end{corollary} 

In the next section we prove the following analogous result.
\begin{theorem} \label{genIDth} For all $n\ge 1$, the  number of ID trees on $[2n]$ is equal to $g_n$.
\end{theorem}
From this we see that the nonmedian Genocchi numbers also  play an important role in the study of the intersection lattice of the homogenized Linial arrangement.
\begin{corollary} \label{gencor}
For all $n \ge 1$,

\begin{equation} \label{geneq} \mu_{\mathcal L(\mathcal H_{2n-1})}(\hat 0, \hat 1) = - g_n.\end{equation}

\end{corollary}

\section{From ID forests to Dumont-like permutations} \label{Dsec}

Our next step is to introduce a class of permutations similar to the Dumont permutations discussed in Section ~\ref{gensec} and then give a bijection between these permutations  and  the ID forests.  This will enable us to express the characteristic polynomial in terms of D-permutations in several ways.

\subsection{D-permutations} 
Let $A$ be a finite subset of $\Z_{> 0}$.  We say $\sigma \in \s_A$ is a {\em D-permutation} on $A$ if 
$i \le \sigma(i) $ whenever $i$ is odd and $i \ge \sigma(i)$ whenever $i$ is even.
We denote by $\mathcal{D}_A$ the set of D-permutations on $A$ and by $\mathcal{DC}_A$ the set of D-cycles on $A$. If $A = [n]$,  we write $\mathcal{D}_n$ and $\mathcal{DC}_{n}$.

\begin{example}
The D-permutations on $[4]$ (in cycle form) are 
\begin{center}$\begin{array}{cccc}
(1)(2)(3)(4) & (1,2) (3)(4) & (1,4) (2) (3) & (3,4) (1) (2)\\
 (1,4,2) (3) & (1,3,4) (2) & (1,2)(3,4) & (1,3,4,2).
\end{array}$\end{center}
The D-cycles  on $[6]$ are:
$$(1,3,5,6,4,2) \qquad (1,4,3,5,6,2) \qquad (1,5,6,3,4,2) .$$

\end{example}

Note that all Dumont permutations are D-permutations, but not conversely. Indeed, a D-permutation can have both even and odd fixed points, while a Dumont permutation can only have odd fixed points.  It follows immediately from the definitions that
$$\mathcal{DC}_{2n} \subseteq \{\mbox{Dumont derange. in } \s_{2n}\} \subseteq  \{\mbox{Dumont 
permut. in } \s_{2n}\} \subseteq \D_{2n}.$$  Recall that the two sets in the middle of this chain are 
enumerated by median Genocchi number $h_{n-1}$ and Genocchi number $g_{n+1}$, respectively.  It turns out that the sets on the ends of the chain are also enumerated by Genocchi and median Genocchi numbers, respectively.   
The count for $\mathcal{DC}_{2n}$ is an easy consequence of   (\ref{otherDumontEq}).  We will see at the end of this section that the count for $\mathcal{D}_{2n}$ can be proved using Hetyei's formula (\ref{heteq2}).  In Section~\ref{DtoSsec} we will prove it without relying on Heyei's result.

\begin{theorem} 
\label{genth}
For all $n \ge 1$,
\begin{enumerate}
\item $|\mathcal{DC}_{2n}| = g_n$
\item $|\mathcal{D}_{2n}| = h_{n}$.
\end{enumerate}
\end{theorem}

\begin{proof}[Proof of (1)]
There is an easy bijection between the set of D-cycles on $[2n]$ and the set $S$ of permutations given in (\ref{otherDumontEq}).  Indeed, we send a D-cycle $(a_1,a_2,\dots,a_{2n})$, where   $a_{2n} = 2n$,  to the permutation $a_1, a_2, \dots, a_{2n-1}$ (in one line notation), which is clearly in  $S$.  Hence the result follows from (\ref{otherDumontEq}).
\end{proof}

\subsection{The bijection} 
Let $\idt$ be the set of ID trees.  For $T \in \idt$, 
 let $\hat T$ be the plane tree obtained  by rooting $T$ at its largest node (which is even when $|T| > 1$) and fixing the following total ordering on the children of each node $v$ of $T$:
\begin{itemize}
\item if $v$ is even, place its children (all of which are odd) in increasing order from left to right
\item if $v$ is odd, place its children (all of which are even) in decreasing order from left to right.
\end{itemize}
Also let $\tilde T$ be the plane tree obtained by rooting $T$ at its smallest node  (which is odd when $|T| > 1$) and  fixing the above total ordering on the children of each node $v$ of $T$.

Recall  from Section~\ref{bondsubsec}, that  $\pw(T)$ denotes the postorder word of a plane tree $T$.    For the ID tree $T$ below, $\pw(\hat T)= 42156378$ and $\pw(\tilde T) = 56378421$.

\begin{center} \includegraphics[scale=0.25]{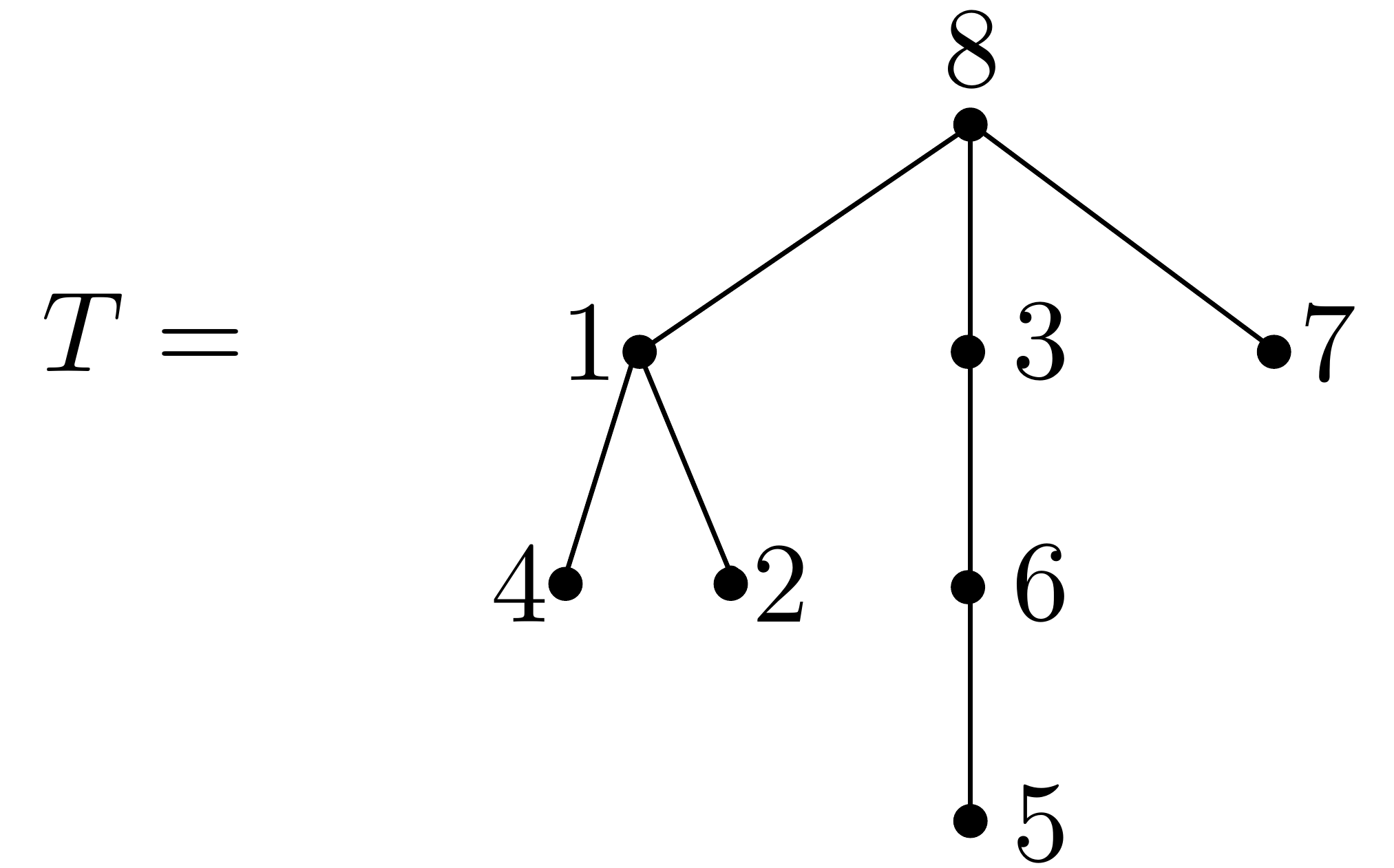} \end{center}

Let $\mathcal W$ be the set of all words  $w_1w_2\dots w_m$, where $m \ge 1$, with distinct letters in $\Z_{> 0}$ that satisfy the following conditions:
 \begin{enumerate}
 \item for all $i \in [m-1]$, if $w_i$ is odd then $w_i < w_{i+1}$
 \item  for all $i \in [m-1]$, if $w_i$ is even then $w_i > w_{i+1}$
 \item if $w_m$ is odd then it is the smallest letter of $w$
 \item if $w_m$ is even then it is the largest letter of $w$.
  \end{enumerate}

\begin{lemma}Let $T \in \idt$. Then $\pw(\hat T) \in \mathcal W$ and $\pw(\tilde T) \in \mathcal W$.  \end{lemma} 

\begin{proof} 
Suppose $T \in \idt$. We prove the result for $\hat T $ and  $\tilde T $  simultaneously by induction on $|T|$. The result is obviously true for $|T|=1$, so assume $|T|>1$.   Let $v_1<\cdots < v_k$ be the children of the root $r$ of $\hat T$.  For each $i$,  the subtree of $\hat T$ rooted at $v_i$ is of the form $\tilde T_i$ for some ID tree $T_i$ by Proposition ~\ref{larsmlprop}.
We have $$\pw(\hat T) = \pw(\tilde T_1) \cdot \,\,\, \cdots \,\,\, \cdot \pw(\tilde T_k) \cdot r,$$
where $\cdot$ denotes concatenation.  Since the last letter of $\pw(\hat T) $ is the root of $\hat T$, it is even and is larger than all the other letters.

By induction on $|T|$, each $\pw(\tilde T_i)$ is in $\mathcal W$. We also know that the last letter of $\pw(\tilde T_i)$ is the root of $\tilde T_i $, which is odd.  Hence to prove that $\pw(\hat T) \in \mathcal W$ we need only show that the last letter $v_i$ of $\pw(\tilde T_i)$ is less than the first letter of $\pw(\tilde T_{i+1})$ for all $i \in [k-1]$.  (Note that the last letter of $\pw(\tilde T_k)$ is less than $r$).  The first letter of $\pw(\tilde T_{i+1})$ is greater than the root $v_{i+1}$ of $\tilde T_{i+1}$.  Since $v_i < v_{i+1}$, we have that $v_i $ is less than the first letter of $\pw(\tilde T_{i+1})$.  Thus $\pw(\hat T) \in \mathcal W$.  An analogous argument yields $\pw(\tilde T) \in \mathcal W$.
\end{proof}

Given a word $w = w_1w_2 \cdots w_m$ over alphabet $\Z_{> 0}$, we say that $w_j$ is a {\em right-to-left minimum} if $w_j < w_i$ for all $i >j$. Similarly we say that $w_j$ is a {\em right-to-left maximum} if $w_j >w_i$ for all $i >j$. For example, the right-to-left minima of $2156437$ are $1,3,7$.  We also let $w^\prime$ denote $w$ with its last letter removed, that is $w^\prime = w_1 \cdots w_{m-1}$.

\begin{lemma}\label{lrmlem} Let $T \in \idt$. Then the children of the root of $\hat T$ are exactly the right-to-left minima of $\pw(\hat T)^\prime$, and the children of the root of $\tilde T$ are exactly the right-to-left maxima of the word $\pw(\tilde T)^\prime$.
\end{lemma}

\begin{proof} Let $v_1<\cdots < v_k$ be the children of the root $r$ of $\hat T$.  Each $v_j$ is less  than all  its descendants in $\hat T$, so for all $i<j$, we have that $v_i$ is less than all descendents of $v_j$. Hence, $v_i$, which is the last letter of the postorder word of the subtree rooted at $v_i$, is less than all subsequent letters of $\pw(\hat T)^\prime$.   Conversely,  if $v$ is not a child of $r$ then it has an odd anscestor $u$, which must be less than $v$.  Since $u$ appears after $v$ in the postorder word, $v$ is not a right-to-left minimum.  The proof for $\tilde T$ is completely analogous. \end{proof}

\begin{lemma} \label{wordlemma} The map $ \pw: \{\hat T : T \in \idt\} \cup \{\tilde T : T \in \idt\} \to \mathcal W $ is a bijection.
\end{lemma}

\begin{proof}To show that the map is a bijection, we construct  a map $$\gamma: \mathcal W \to \{\hat T : T \in \idt \} \cup 
\{\tilde T : T \in \idt \}$$  and show that it is the inverse of $w$.

We define $\gamma$ recursively.  Let $w =  w_1\cdots w_m \in \mathcal W$.  If $m=1$, let 
 $ \gamma(w)$ be the tree with the single node $w_1$. 

Now suppose $m >1$.  To define $\gamma$, we need two cases:

{\bf Case 1:} Assume  $w_m$ is even. Then it is the  largest letter of $w$. Let $j_1<\cdots<j_k$ be  such that $w_{j_1} \cdots, w_{j_k}$ are the right-to-left minima of the word $w^\prime:= w_1w_2\dots w_{m-1}$. 
Note that $j_k ={m-1}$.  For each $i = 1,\dots,k$, let $\alpha_i(w) $ be the segment 
$w_{j_{i-1}+1}w_{j_{i-1}+2} \cdots w_{j_i}$, where $j_0 = 0$.  Hence $w^\prime$ is the concatenation of the words 
$\alpha_1(w) , \alpha_2(w),\dots,\alpha_k(w)$. 

For example, given the word $w=21564378$ in $\mathcal W$, we have, $w' = 2156437$, and $w'$ breaks up into $\alpha_1(w) = 21$, $\alpha_2(w) = 5643$, and $\alpha_3(w) = 7$.

Since for all $i$,  $\alpha_i(w)$  is a segment of $w \in \mathcal W$, it satisfies conditions (1) and (2) of the definition of $\mathcal W$.
We claim that conditions (3) and (4) hold as well.   Indeed, the last letter $w_{j_{i}}$ of $\alpha_i(w) $  is odd since it is a right-to-left minimum of $w^\prime$ and is therefore is followed by a larger letter in $w$.  The last letter  of $\alpha_i(w)$ is smaller than all the other letters of $\alpha_i(w)$ since  none of the other letters are  right-to-left minima.  Hence (3) holds and (4) holds vacuously. Thus $\alpha_i(w) \in \mathcal W$.   

Now, for each $i$, we can recursively apply $ \gamma$ to  $\alpha_i(w)$ to obtain a plane tree  $ \gamma(\alpha_i(w))$  in $\{\tilde T : T \in \idt\}$.  Let $\gamma(w)$ be the plane tree constructed as follows: 
\begin{enumerate}
\item The root is $w_m$, which is even and largest.
\item The children of $w_m$ are the right-to-left minima $w_{j_1}< \cdots <w_{j_k}$, which are odd.
\item The subtree rooted at $w_{j_i}$ is $ \gamma(\alpha_i(w))$ for each $i$.
\end{enumerate}
Clearly $\gamma(w) \in  \{\hat T : T \in \idt\}$. 

{\bf Case 2:} Assume  $w_m$ is odd.
An analogous argument allows us to decompose $w = w_1 \cdots w_{m-1} \in {\mathcal W}$ into  segments $\alpha_i(w) \in \mathcal W$, whose last letter is the $i$th  right-to-left maximum of $w$.   We define  $\gamma(w)$ to be the plane tree constructed as follows: 
\begin{enumerate}
\item The root is $w_m$, which is odd.
\item The children of $w_m$ are the right-to-left maxima $w_{j_1}> \cdots >w_{j_k}$, which are even
\item The subtree rooted at $w_{j_i}$ is $\gamma(\alpha_i(w))$ for each $i$.
\end{enumerate} 
We have $\gamma(w) \in  \{\tilde T : T \in \idt\}$.

Now that we have shown that $\gamma$ is well-defined, it remains to check that the maps $\pw$ and $\gamma$ are inverses of each other. We  prove $\gamma(\pw(\hat T) )= \hat T$  and $\gamma(\pw(\tilde T) )= \tilde T$,
for all $T \in \idt$, 
 by induction on $|T|$. If $|T| = 1$, this is clear, so suppose $|T| > 1$.     
 
 {\bf Case 1:}  $\hat T$. By Lemma~\ref{lrmlem}, the children $ v_1<v_2 <\dots < v_k$ of the root $r$ of $\hat T$ are exactly the right-to-left minima of   the word $\pw(\hat T)^\prime = \pw(\tilde T_1) \cdot \pw(\tilde T_2) \cdot \,\, \cdots \,\, \cdot  \pw(\tilde T_k) $, where $\tilde T_i$ is the subtree of $\hat T$ rooted at $v_i$.
Since $v_i$ is the last letter of $\pw(\tilde T_i)$, we have $\alpha_i( \pw(\hat T)) = \pw(\tilde T_i )$ for all $i$.  
By induction, $\gamma(\alpha_i( \pw(\hat T))) =  \gamma( \pw(\tilde T_i )) = \tilde T_i .$

Now  we have that $\gamma ( \pw(\hat T))$ is the  plane tree whose root is the last letter of $ \pw(\hat T)$, which is $r$.  The children of $r$ are the right-to-left minima of $\pw(\hat T)^\prime$ which are $v_1<\dots< v_k$.  The subtree rooted at $v_i$ is $ \gamma(\alpha_i( \pw(\hat T)))$ which is $\tilde T_i$.  Hence  $\gamma ( \pw(\hat T)) = \hat T$, as desired.    

 {\bf Case 2:}  $\tilde T$. A completely analogous argument gives $ \gamma ( \pw(\tilde T)) = \tilde T$. 
 
 Hence $ \gamma \circ \pw$ is the identity map on $\{\hat T : T \in \idt\} \cup 
\{\tilde T : T \in \idt\}$.  A similar argument can be used to prove that 
the other composition $ \pw \circ \gamma$ equals the identity map on  $\mathcal W$.
 \end{proof}
 
For any finite subset $A$ of  $\Z_{> 0}$, let $\idt_A = \{T \in \idt : \mbox{the node set of $T$ is $A$}\}$ and let $\mathcal W_A = \{w \in \mathcal W : \mbox{the  set of letters of $w$ is $A$}\}$. Recall that $\mathcal{DC}_A$ is the set of D-cycles on $A$.
 
  \begin{theorem} \label{IncDecBijection} For   any finite suubset $A$ of  $\Z_{> 0}$, let $\psi:\idt_A \to \mathcal{DC}_{A}$ be the map defined by letting $\psi(T)$ be the cycle $(\pw(\hat T))$ in $\sg_A$.  Then $\psi$ is a well defined bijection.  Consequently $|\idt_A| = |\mathcal{DC}_{A}|$.
 \end{theorem}
 
 \begin{proof} Consider the following maps: 
 \begin{enumerate}
\item  the map from $ \idt_A$ to  $\{\hat T: T \in \idt_A\}$ defined by $T \mapsto \hat T$
\item  the map from  $\{\hat T: T \in \idt_A\}$ to $\{w \in \mathcal W_A : \mbox{ last letter of $w$ is largest} \}$ defined by $\hat T  \mapsto \pw(\hat T)$
\item the map from $\{w \in \mathcal W_A : \mbox{last letter of $w$ is largest} \}$ to $\mathcal{DC}_{A}$ defined by $w \mapsto (w)$.
\end{enumerate} 
It is clear that the first map is a bijection. It follows from Lemma~\ref{wordlemma} that the second is also bijection.  The third map  is  a well defined bijection since  the last letter of $w$ must be even.   Since $\psi$ is the composition of  three bijections, it is also a bjection.
  \end{proof}

An analogous argument can be used to show that the map $\idt_A \to \mathcal{DC}_{A}$ that takes $T$ to the cycle $(\pw(\tilde T))$ is also a well defined bijection.  The following result shows  that this map is, in fact, identical to $\psi$.

\begin{proposition} \label{cycleslem} For all $T \in \idt_A$,   the cycles $(\pw(\hat T))$ and $(\pw(\tilde T))$ in $\sg_A$ are equal.  
\end{proposition}

\begin{proof} Let $s$ be the smallest element of $A$ and $t$ the largest element.  Then $s$ is the root of $\tilde T$ and  $t$ is its leftmost child.  Let $t_1<\dots,<t_k$ be be the children of $t$ in $\tilde T$ and  let $t>s_1 >\dots > s_j$ be the children of $s$.   For each $i \ge 1$, let $T_i$ be the subtree  of $\tilde T$  rooted at $t_i$ and let $S_i$ be the subtree of $\tilde T$ rooted at $s_i$.
Then $\pw(\tilde T)$ is the concatenation  $$\pw(T_1) \cdots \pw(T_k) \cdot t \cdot \pw(S_1) \cdots \pw(S_j) \cdot s.$$

Now let us look at $\pw(\hat T)$. Clearly $t$ is the root of $\hat T$ and $s$ is its leftmost child.  The children of $t$ are $s< t_1<\dots < t_k$.  The children of $s$ are $s_1 < \dots < s_j$.  For each $i \ge 1$, the  subtree  of $\hat T$  rooted $s_i$ is $S_i$ and the subtree rooted at $t_i$ is $T_i$.
Then $\pw(\hat T)$ is  the concatenation  $$\pw(S_1) \cdots \pw(S_j) \cdot s \cdot \pw(T_1) \cdots \pw(T_k) \cdot t .$$  From this we see that 
$\pw(\tilde T)$ and $\pw(\hat T)$ are related by a cyclic shift.  Hence the cycles are the same.
\end{proof}

The {\em cycle support} of $\sigma \in \s_n$ is the partition $\cyc(\sigma) \in \Pi_n$ whose blocks are comprised of the elements of the cycles of $\sigma$.  For example,
$$\cyc((1,7,2,4)(5)(6,8,9,3)) = 1247|5|3689.$$
The bijection in Theorem \ref{IncDecBijection} extends to a bijection between  ID forests and D-permutations.  Under this bijection, the blocks of the cycle support of the image of an ID forest are   the node sets of the trees  of the ID forest.  Hence we have the following consequence of Theorem \ref{IncDecBijection}.

\begin{corollary} \label{IncDecBijectionCor} Let $\pi$ be a partition of a finite subset of  $Z_{>0}$.  Then the ID forests whose trees have node sets equal to the blocks of $\pi$ are in bijection with the D-permutations  whose cycle support is $\pi$.
\end{corollary}

As a consequence of Theorem \ref{doubth} and Corollary ~\ref{IncDecBijectionCor}, we have the following result.

\begin{theorem} \label{mobth}
Let $A$ be a finite subset of $\Z_{>0}$.  For  all $\pi \in  \Pi_{\Gamma_{A}}$,
$$(-1)^{|\pi|}\mu_{\Pi_{\Gamma_{A}}}(\hat{0},\pi) = |\{\sigma \in \mathcal{D}_{A} \st \cyc(\sigma) = \pi\} |.$$ 
Consequently, \begin{equation} \label{chardumeq}  \chi_{\mathcal L(\Pi_{\Gamma_A})}(t) = 
\sum_{k=1}^{2n} s_D(A,k) t^{k-1},\end{equation}
where $(-1)^{k}s_D(A,k)$ is equal to  the number of D-permutations on $A$ with exactly $k$ cycles.
\end{theorem}

Equation~(\ref{chardumeq}) is equivalent to the following result on the chromatic polynomial  ${\rm ch}_G(t)$ of a general Ferrers graph $G$ (see (\ref{chromchareq}) and Remark~\ref{ferrem}).
 \begin{corollary} For all finite subsets $A$ of $\Z_{>0}$,
$${\rm ch}_{\Gamma_A} (t) = \sum_{k=1}^{2n} s_D(A,k) t^{k} .$$
\end{corollary}

By Theorem~\ref{bondth}, we  have now proved (\ref{intorchardumeq}), which is restated here.
\begin{corollary} For all $n \ge 1$,
$$ \chi_{\mathcal L(\mathcal H_{2n-1})} = 
\sum_{k=1}^{2n} s_D(2n,k) t^{k-1},$$
where $(-1)^{k} s_D(2n,k)$  is equal to  the number of D-permutations  on $[2n]$ with exactly $k$ cycles.
\end{corollary}

\begin{corollary} \label{mucor} For all $n \ge 1$,
\begin{align*} \chi_{\mathcal L(\mathcal H_{2n-1})}(-1) &= - |\mathcal{D}_{2n}| \\
\mu_{\mathcal L(\mathcal H_{2n-1})}(\hat 0, \hat 1) &= - |\mathcal{DC}_{2n}|.
\end{align*}
\end{corollary}

Note that it follows from Corollary~\ref{mucor} that Part (2) of  Theorem~\ref{genth} is equivalent to Hetyei's formula (\ref{heteq2}).  So Part (2) is now proved.
  In the next section we will prove Part (2) without relying on (\ref{heteq2}), thereby recovering Hetyei's result.

Note also that Theorem~\ref{genIDth} is now proved since it follows immediately from  Theorems~\ref{genth}(1)  and~\ref{IncDecBijection}.

\subsection{Alternative  formulas} \label{altsec} We now give two alternative ways of expressing the characteristic polynomial in terms of D-permutations (equivalently  ID forests), which  will be used in Section~\ref{genfnsec}.  As a consequence of these formulas, we get a decomposition of the Genocchi numbers and the median Genocchi numbers into a sum of powers of $2$.  

\begin{theorem} \label{chiDlem} For all $n \ge 1$,
\begin{equation} \label{char1eq} \chi_{\Pi_{\Gamma_{2n}}}(t) = (t-1) \sum_{\sigma\in\mathcal{D}_{2n-2}}(-t)^{\#\{\text{\rm{even fixed points of }}\sigma\}}(1-t)^{\#\{\text{\rm {other cycles of }}\sigma\}}.\end{equation}
\end{theorem} 

\begin{proof}

Let $\idf_V$ be the set of ID forests on node set $V$. For $V=[2n]$, Theorem~\ref{doubth}   can be restated as

$$\chi_{\Pi_{\Gamma_{2n}}}  (t) = - \sum_{F \in \idf_{[2n]}} (-t)^{m(F)-1}  ,$$
where  $m(F)$ denotes the number of trees of the forest $F$.
For each $F \in \idf_{[2n]}$, let $F^\prime$ be the forest obtained by removing node $[2n]$ from $F$.  Clearly $F^\prime  \in \idf_{[2n-1]}$. 
We thus have $$\chi_{\Pi_{\Gamma_{2n}}}  (t) = -\sum_{G \in \idf_{[2n-1]}} \sum_{\scriptsize \begin{array}{c} F \in \idf_{[2n]} \\ F^\prime = G \end{array}} (-t)^{m(F)-1}  .$$

Fix $G \in \idf_{[2n-1]}$ and let $T$ be a tree of the forest $G$.   
For all $F \in \idf_{[2n] }$ for which $F^\prime = G$,  the tree $T$ is either a tree of the forest $F$  or it is attached  to $2n$ in $F$.  Hence to obtain a forest  $F \in \idf_{[2n] }$ for which $F^\prime = G$,  we have the option of not attaching $T$ to $2n$ or attaching $T$ at its smallest node, which can  be done if and only if $T$ is not an even single node.  Hence  if $T$ is not an even single node then it contributes a factor of $1-t$ to the inner sum, and if $T$ is an  even single node then  it contributes a factor of $-t$ to the inner sum.  
It follows that 
$$ \sum_{\scriptsize \begin{array}{c} F \in \idf_{[2n]} \\ F^\prime = G \end{array}} (-t)^{m(F)-1} =  (-t)^{\#\{\text{even isolated nodes of }G\}}(1-t)^{\#\{\text{other trees of }G\}}.$$ We now have  

\begin{align*}\chi_{\Pi_{\Gamma_{2n}}}  (t) &=  -\sum_{G \in \idf_{[2n-1]}}  (-t)^{\#\{\text{even isolated nodes of }G\}}(1-t)^{\#\{\text{other trees of }G\}}\\
&= -(1-t)  \sum_{G \in \idf_{[2n-2]}}  (-t)^{\#\{\text{even isolated nodes of }G\}}(1-t)^{\#\{\text{other trees of }G\}},  
\end{align*}
with the second equality following from the fact that  every $ G \in \idf_{[2n-1]}$ consists of the isolated node $2n-1$ and an ID forest on $[2n-2]$.
The result now follows from  Corollary~\ref{IncDecBijectionCor}.
\end{proof}

\begin{theorem} \label{QuotCharLem} For all $n \ge 2$,
\begin{equation} \label {char2eq}   \chi_{\Pi_{\Gamma_{2n}}}(t)  = (t-1)^3 \sum_{\sigma\in\mathcal{D}_{2n-4}}(1-t)^{\#\{\text{fixed points of }\sigma\}}(2-t)^{\#\{\text{other cycles of }\sigma\}}.\end{equation}
\end{theorem}

\begin{proof} 
 Each $F \in \idf_{ [2n]}$ can be obtained from a forest in  $\idf_{ [2n] \setminus \{2\}}$ either by adding an isolated node $2$ or by attaching the node $2$ to $1$.   This implies that 
\begin{equation}\label{factor1eq} \chi_{\Pi_{\Gamma_{[2n]}}}= (1-t) \chi_{\Pi_{\Gamma_{[2n]\setminus\{2\}}}}.\end{equation}
Similarly, \begin{equation} \label{factor2eq}  \chi_{\Pi_{\Gamma_{[2n]\setminus\{2\}}}} = (1-t)\chi_{\Pi_{\Gamma_{[2n]\setminus\{2,2n-1\}}}}.\end{equation}

Now let $\overline{ \idf}_{ [2n] \setminus \{2,2n-1\}}$ be the subset of  $\idf_{ [2n] \setminus \{2,2n-1\}}$ consisting of forests in which the nodes $1$ and $2n$ are in distinct trees.  Since each forest in $\idf_{ [2n] \setminus \{2,2n-1\}}$ is either a forest in  $\overline{\idf}_{ [2n] \setminus \{2,2n-1\}}$ or is obtained from one in $\overline{\idf}_{ [2n] \setminus \{2,2n-1\}}$ by adding an edge between the nodes $1$ and $2n$, we have by Theorem~\ref{doubth} that
\begin{equation} \label{factor3eq} \chi_{\Pi_{\Gamma_{[2n]\setminus \{2,2n-2\}}}} (t) = -(1-t)  \sum_{F \in \overline{ \idf}_{[2n]\setminus \{2,2n-2\}}} (-t)^{m(F)-1} .\end{equation}

For each $F \in \overline{\idf}_{ [2n] \setminus \{2,2n-1\}}$, let $\hat F$ be the forest obtained by removing the nodes $1$ and $2n$ from $F$.  Clearly $\hat F  \in \idf_{\{3,\dots,2n-2\}}$. 
It therefore follows from (\ref{factor1eq}), (\ref{factor2eq}), and (\ref{factor3eq}) that  
\begin{align} \label{factor4eq} \chi_{\Pi_{\Gamma_{[2n]}}}(t) &= (t-1)^3 \sum_{F \in \overline{\idf}_{ [2n] \setminus \{2,2n-1\}}} (-t)^{m(F)-1} 
\\ \nonumber &= (t-1)^3 \sum_{G \in \idf_{\{3,\dots,2n-2\}}} \sum_{\scriptsize \begin{array}{c} F \in  \overline{\idf}_{ [2n] \setminus \{2,2n-1\}} \\ \hat F = G \end{array}} (-t)^{m(F)-1}  .\end{align}

Fix $G \in \idf_{\{3,\dots,2n-2\}}$ and let $T$ be a tree of the forest $G$.   For all $F \in  \overline{\idf}_{ [2n] \setminus \{2,2n-1\}}$ for which $\hat F = G$,  the tree $T$ is either a tree of the forest $F$  or it is attached  to either $1$ or $2n$ in $F$.  If $T$ is a single node it can  be attached to only one of the nodes in $\{1,2n\}$ depending on its parity.  Hence singleton trees of $G$ contribute a factor of $1-t$ to the inner sum.  If $T$ has more than one node then there are exactly two ways to attach it;  with an edge between its largest node and $1$,  or  with an edge between its smallest  node and $2n$.  Thus, nonsingleton trees of $G$ contribute a factor $2-t$ to the inner sum.  
It follows that 
$$ \sum_{\scriptsize \begin{array}{c} F \in  \overline{\idf}_{ [2n] \setminus \{2,2n-1\}} \\ \hat F = G \end{array}} (-t)^{m(F)-1}  = (1-t)^{\#\{\text{isolated nodes of }G\}}(2-t)^{\#\{\text{other trees of }G\}}.$$

We now have   
\begin{align*} \chi_{\Pi_{\Gamma_{2n}}}(t)  &= (t-1)^3 \sum_{G \in \idf_{\{3,\dots,2n-2\}}}  (1-t)^{\#\{\text{isolated nodes of }G\}}(2-t)^{\#\{\text{other trees of }G\}} \\
&=  (t-1)^3\sum_{G \in \idf_{\{1,\dots,2n-4\}}}  (1-t)^{\#\{\text{isolated nodes of }G\}}(2-t)^{\#\{\text{other trees of }G\}} \\
&= (t-1)^3 \sum_{\sigma\in\mathcal{D}_{2n-4}}(1-t)^{\#\{\text{fixed points of }\sigma\}}(2-t)^{\#\{\text{other cycles of }\sigma\}},\end{align*}
with the last equality following from (\ref{IncDecBijectionCor}).  \end{proof}

Note that when $t$ is set equal to $0$, equation (\ref{char1eq}) reduces to (\ref{geneq}).  Indeed, the only  terms to survive on right hand side of (\ref{char1eq}) are the ones corresponding to D-permutations with no even fixed points.  But these are the Dumont permutations, which were used in Section~\ref{gensec}  to  define the Genocchi numbers.  

In the following corollary, we obtain the decomposition of the median Genocchi numbers by setting $t=-1$ in (\ref{char1eq}) and using (\ref{heteq2}). We obtain the decomposition of the Genocchi numbers by setting $t=0$ in (\ref{char2eq}) and using (\ref{geneq}).

\begin{corollary}  For all $n \ge 2$, 
\begin{align*} h_n &= \sum_{j = 1}^{n-1} h_{n-1, j} 2^{j+1} \\
g_{n+1} &= \sum_{j = 0}^{n-1} g_{n-1, j} 2^{j},
\end{align*}
where $h_{n,j}$ is the number of D-permutations on $[2n]$ with exactly $j$ cycles that are not even fixed points and $g_{n,j}$ is the number of D-permutations on $[2n]$ with exactly $j$ cycles that are not  fixed points.
\end{corollary}

We expect that the coefficients $g_{n-1,j}$ are equal to the coefficients in   Sundaram's  decomposition \cite[Proposition 3.5]{sund} of the Genocchi numbers into powers of 2.

\subsection{Reducing the characteristic polynomial} \label{redsec}
From (\ref{char2eq}) we can see that $ \chi_{\Pi_{\Gamma_{2n}}}(t) $ is divisible by $(t-1)^3$.  In this subsection we give combinatorial and geometric interpretations of the polynomial $ (t-1)^{-3} \chi_{\Pi_{\Gamma_{2n}}}(t) $.

    Let $L_{2n}$ be the geometric semilattice (see \cite{geometric_semilattices}) obtained from $\Pi_{\Gamma_{[2n]\setminus \{2,2n-2\}}}$ by removing the order filter generated by  the atom  of $\Pi_{\Gamma_{[2n]\setminus \{2,2n-2\}}}$ whose only nonsingleton block is $\{1,2n\}$. That is,
  $$ L_{2n} = \{ \pi \in \Pi_{\Gamma_{[2n]\setminus \{2,2n-2\}}} :  1 \mbox { and } 2n \mbox{ are in distinct blocks of  }\pi\}.$$  
  Geometrically, $L_{2n}$ is the intersection semilattice of the affine hyperplane arrangement obtained by deconing the graphic arrangement of $\Gamma_{[2n]\setminus \{2,2n-2\}}$.  Indeed,  let $\mathcal P_{2n}$ be the affine arrangement obtained by intersecting the graphical arrangement  of the graph $\Gamma_{[2n]\setminus\{2,2n-1\} }$ with the affine hyperplane  $x_1 = x_{2n}+1$.  It is easy to see that the intersection semilattice  $\mathcal L(\mathcal P_{2n})$ of $\mathcal P_{2n}$ is $L_{2n}$.

  From (\ref{factor4eq}),
we can see that
 \begin{equation} \label{chiLeq} \chi_{L_{2n} }(t)= (t-1)^{-3} \chi_{\Pi_{\Gamma_{2n}}}(t).\end{equation}
  By dividing both sides of (\ref{char2eq}) by $(t-1)^3$ and then setting $t=1$, the median Genocchi numbers appear  again.

\begin{corollary} \label{quotcor} For all $n \ge 3$, 
\begin{align} \label{quot0eq}  \chi_{L_{2n}}(0) &= g_{n }
\\ \label{quot1eq}
 \chi_{L_{2n}}(1) & = h_{n-3}
  \end{align}
 Consequently, the number of bounded regions of $\mathcal P_{2n}$ is $h_{n-3}$.
\end{corollary}

\begin{proof}  Equation~(\ref{quot0eq}) is an immediate consequence of (\ref{chiLeq}) and Corollary~\ref{gencor}.

It follows from Theorem~\ref{QuotCharLem} that  $   \chi_{L_{2n}}(1)$  equals the number of D-permutations on $[2n-4]$ with no fixed points.  
But these are precisely the Dumont derangements on $[2n-4]$, which by the definition of the median Genocchi numbers  given in Section~\ref{gensec}, is equal to $h_{n-3}$.  Hence (\ref{quot1eq}) holds.

The consequence follows from  another well known result of Zaslavsky  \cite{Facing_Up_Arrangements}, namely that the number of bounded regions of any affine arrangement $\mathcal A$ is equal to $|\chi_{L(\mathcal A)}(1)|$.
\end{proof}

\section{From D-permutations to surjective staircases} \label{DtoSsec}

In this section we prove   (\ref{introgenchareq}) along with a similar formula for the generating function of the characteristic polynomial of the homogenized Linial arrangement. The proofs rely on the theory of surjective staircases introduced by Dumont \cite{Interpretations_Combinatoires}. We first construct a bijection from the D-permutations on $[2n]$ to a certain class of surjective staircases known to be enumerated by $h_n$, thereby proving Theorem~\ref{genth} (2) and completing our path to Hetyei's formula 
(\ref{heteq2}).

\subsection{Surjective staircases}
An \emph{excedent function} is a map $f:[m] \to [m]$ such that for all $i \in [m]$, $f(i) \geq i$. It is convenient to visualize excedent functions by associating them    with fillings of the Ferrers diagram of shape $(m,m-1,\dots,1)$. We write the diagram in English notation with the row lengths  decreasing from top to bottom. 
The rows are labeled from top to bottom with the numbers $m$ down to $1$, and the columns are labeled from left to right with the numbers $1$ up to $m$.
 Let $C_{i,j}$ be the cell in the row labeled $i$ and the column labeled $j$. The tableau $T(f)$ corresponding to excedent function $f$ has an $X$ in cell $C_{i,j}$ whenever $f(j) = i$ and blanks in the other cells.  For example, the tableau in the figure below represents the excedent function $f:[4] \to [4]$ with $f(1) = 1, f(2) = 4, f(3) = 5, f(4) = 4, f(5) =5$.

$$  \begin{ytableau}
\none & \none[1] & \none[2] & \none[3] & \none[4] & \none[5] \\ 
\none[5] &\; & \; & X & \; & X \\
\none[4] & \; & X & \; & X \\
\none[3] & \; & \; & \;  \\
\none[2] & \; & \; \\
\none[1] & X
\end{ytableau} $$

Clearly this correspondence is a bijection between the set $\mathcal X_m$  of excedent functions $f: [m] \to [m]$ and   the set of tableaux $T$ of shape  $(m,m-1,\dots,1)$ filled with $X$'s and blanks so that each column has exactly one $X$.  The image of $f$ corresponds to the set of nonempty rows of $T(f)$, that is the rows with an $X$.  The {\em fixed points} of $f$ are the $j \in [m]$ for which $f(j) = j$.  Clearly the fixed points of $f$ correspond to the rows of $T(f)$ that have an $X$ in the rightmost cell of the row.   
An \emph{isolated fixed point} of $f$ is a fixed point $j$ of $f$ such that $f^{-1}(j) = \{j\}$.  Clearly the isolated fixed points correspond to the rows of $T(f)$ that have only one $X$, which is in the rightmost cell of the row.

It is not difficult to see  that $|\mathcal X_m| = m!$.  For our purposes, we will need a bijection from $\sg_m$ to $\mathcal X_m$  constructed in \cite{Derangements_Genocchi}.

\begin{proposition}[Dumont and Randrianarivony {\cite[Proposition 1.3]{Derangements_Genocchi}}] \label{DuRaProp}
There exists a bijection $\tau: \sg_m \to \mathcal X_m$ such that for all $\sigma \in \sg_m$ and $j \in [m]$, the following properties hold:
\begin{enumerate}

\item $j$ is a cycle maximum of $\sigma$ if and only if it is a fixed point of $\tau(\sigma)$,
\item $j$ is a fixed point of $\sigma$ if and only if it is an isolated fixed point of $\tau(\sigma)$,
\item  $\sigma(j) \leq j$ if and only if $j$ is in the image of $\tau(\sigma)$.

\end{enumerate}
\end{proposition}

An excedent function $f \in \mathcal X_{2n}$ is said to be a {\em surjective staircase}\footnote{These are also known as surjective pistols in the literature.}  if its image is $\{2,4,\dots, 2n\}.$  Since the odd labeled rows in  the tableau $T(f)$ corresponding to a surjective staircase $f$ are empty, we can delete these rows from the tableau.  Let $T^\prime(f)$ denote the tableau obtained from $T(f)$ by deleting the odd labeled rows.  For example, let $f:[10] \to [10]$ be defined by
$$f(1) =2, f(2) = 8, f(3) = 4, f(4) = 4, f(5) = 10,$$
$$ f(6) = 6, f(7) = 10, f(8) = 8, f(9) = 10, f(10) = 10.$$ Then $T^\prime(f)$ is the tableau given in the figure below.

$$\begin{ytableau}
\none & \none[1] & \none[2] & \none[3] & \none[4] & \none[5] & \none[6] & \none[7] & \none[8] & \none[9] & \none[10]\\ 
\none[10] & \; & \; & \; & \; & X &\; & X & \; & X & X\\
\none[8] & \; & X & \; & \; & \; & \; & \; & X\\
\none[6] & \; & \; & \; & \; & \; & X\\
\none[4] &\; & \; & X & X\\
\none[2] & X & \;  
\end{ytableau}$$

Let $\mathcal E_{2n}$ be the set of surjective staircases on $2n$. For $f  \in \mathcal E_{2n}$, we say that $j$ is a {\em maximum} of $f$ if $j \in [2n-2]$ and $f(j) = 2n$, that is, if $f$   achieves its maximum value at $j$ and $j \ne 2n-1,2n$.  So $j\in [2n-2]$ is a maximum of  $f$ if, in $T^\prime(f)$, the top row    has an $X$ in column $j$ .  For the surjective staircase shown in the  figure above, the maxima  are $\{5,7\}$. 

\subsection{The bijection}

We use the Dumont-Randrianarivony bijection (Proposition~\ref{DuRaProp}) to prove the following result.

\begin{lemma}\label{bijectionProp}
There is a bijection $$\phi:  \mathcal{D}_{2n} \to  \{ f \in \mathcal{E}_{2n+2}: f \mbox { has  no even maxima} \}$$ such that for all $\sigma \in \mathcal{D}_{2n}$ and $j \in [2n]$, the following properties hold:
\begin{enumerate}
\item $j$ is an even cycle maximum of $\sigma$ if and only if it is a fixed point of $\phi(\sigma)$,
\item $j$ is an even fixed point of $\sigma$ if and only if it is an isolated fixed point of $\phi(\sigma)$,
\item  $j$ is an odd fixed point of $\sigma$ if and only if it is an odd maximum of $\phi(\sigma)$.
\end{enumerate} 
\end{lemma}

\begin{proof}
Let $ \mathcal G_{2n} $ be the set of all excedent functions  $g \in \mathcal X_{2n}$ that satisfy  
\begin{itemize}
\item the image of $g$ contains $\{2,4,\dots, 2n \}$ 
\item for all $j \in [2n]$, $g(j)$ is odd if and only if $j$ is an odd isolated fixed point of $g$.
\end{itemize}  
 Let $\tau$ be the bijection of Proposition~\ref{DuRaProp}.  We claim that  $\tau(\mathcal D_{2n}) =  \mathcal G_{2n}$.  Indeed, let $\sigma \in \mathcal{D}_{2n}$. The image of the excedent function $\tau(\sigma)$  is $\{2,4,\dots,2n\} \cup \{\text{odd fixed points of } \sigma\}$ by property (3) of Proposition~\ref{DuRaProp}, since $\sigma$ is a D-permutation.  Hence $\tau(\sigma)(j)$ is odd if and only if it is a fixed point of $\sigma$.  By property (2) of  Proposition~\ref{DuRaProp}, the odd fixed points of $\sigma$ are the odd isolated fixed points of $\tau(\sigma)$.

To prove the result we construct a bijection $$\gamma: \mathcal G_{2n} \to \{ f \in \mathcal E_{2n+2}: f \mbox{ has no even maxima}
\}.$$  Then we show that the map $\gamma \circ \tau$ restricted to $\mathcal D_{2n}$ has the desired properties.

Let $g \in \mathcal G_{2n}$.  We define an excedent function $f:[2n+2] \to [2n+2]$ by
$$f(j) = \begin{cases}   g(j), & j \in [2n] \text{ and } g(j) \text{ is even}    \\ 
2n+2, &  j \in [2n] \text{ and } g(j) \text{ is odd}   \\
2n+2, &  j\in \{2n+1,2n+2\} . \end{cases}$$
It is clear that $f \in \mathcal E_{2n+2}$.  We claim that $f$ has no even maxima.  Indeed, the  maxima of $f$ are the $j$ for which  $g(j)$ is odd.  By the definition of $\mathcal G_{2n}$,  such a $j$ must be  a fixed point of $g$,  so all the maxima of $f$ are  odd.    We can now let $$\gamma(g) =f.$$

It is helpful to visualize the map $\gamma$ in terms of tableaux.  First  $\gamma$  adds an empty row labeled $2n+2$ of length $2n+2$ to the top of the tableau $T(g)$, where $g \in \mathcal G_{2n}$.  An $X$ is placed in both of its two rightmost cells.  An $X$ in an odd labeled row of $T(g)$ represents  an isolated fixed point of $g$, so it is in the rightmost  cell of its row.   Next $\gamma$   slides each such $X$ to the top of its  column.    At this point, the odd labeled rows become empty.   We remove these rows to get $T^\prime(\gamma(g))$.  The figure below illustrates the map $\gamma$.

\vspace{.1in}
\begin{ytableau} \none & \none[1] & \none[2] & \none[3] & \none[4] & \none[5] & \none[6] & \none[7] & \none[8]\\
\none[8] & \; & X & \; & \; & \; & \; & \; & X\\
\none[7] & \; & \; & \; & \; & \; & \; & *(yellow) X\\
\none[6] & \; & \; & \; & \; & \; & X\\
\none[5] & \; & \; & \; & \; &*(yellow)  X\\
\none[4] & \; & \; & X & X\\
\none[3] & \; & \; & \;\\
\none[2] & X & \;\\
\none[1] & \;\end{ytableau}

\vspace{-1.1in} \hspace{2.2in} $ \xrightarrow{\gamma}$

\vspace{-1in}
 \hspace{2.6in}\begin{ytableau}
\none & \none[1] & \none[2] & \none[3] & \none[4] & \none[5] & \none[6] & \none[7] & \none[8] & \none[9] & \none[10]\\ 
\none[10] & \; & \; & \; & \; &*(yellow) X &\; &*(yellow) X & \; & X & X\\
\none[8] & \; & X & \; & \; & \; & \; & \; & X\\
\none[6] & \; & \; & \; & \; & \; & X\\
\none[4] &\; & \; & X & X\\
\none[2] & X & \;
\end{ytableau}

 \vspace{.8in}

We claim $\gamma$ is bijective.  Indeed, given   a surjective staircase $f \in \mathcal{E}_{2n+2}$ with no even maxima, we can obtain an excedent function $g: [2n] \to [2n]$ by inserting empty odd labeled rows between the even labeled rows in $T^\prime(f)$ and then sliding the  $X$'s from  the top row down to the bottom of their columns.  After deleting the top row  we obtain  $T(g)$.  Clearly $g \in \mathcal G_{2n}$.  It is not difficult to see that the sliding down process and  sliding up process are inverses of each other.

The sliding process also makes it easy to see that for all  $g \in \mathcal G_{2n}$ and $j \in [2n]$, 
\begin{enumerate}
\item[(A)] $j$ is an even fixed point of $g$ if and only if $j$ is a fixed point of $\gamma(g)$,
\item[(B)] $j$ is an even isolated fixed point of $g$ if and only if $j$ is an isolated fixed point of $\gamma(g)$,
\item[(C)] $j$ is an odd isolated fixed point of $g$ if and only if $j$ is a maximum of $\gamma(g)$.
\end{enumerate}
Now let $\phi $ be the composition  $\gamma \circ \tau |_{\mathcal D_{2n}}$.  Clearly in combination with condition (1) of Proposition~\ref{DuRaProp}, condition (A)   implies condition (1) of the desired result.  Similarly, in combination with condition (2) of Proposition~\ref{DuRaProp}, conditions (B) and (C) imply conditions (2) and (3), respectively, of the desired result. 
\end{proof}

We are now ready to give a proof of Theorem \ref{genth} (2) that does not rely on Hetyei's formula (\ref{heteq2}).   This proof completes  our path to (\ref{heteq2}).

\begin{proof}[Proof of Theorem \ref{genth}, Part (2)] 
 By Lemma \ref{bijectionProp}, we need only show that 
 \begin{equation} \label{Bpolyeq} | \{ f \in \mathcal{E}_{2n+2}: f \mbox{ has  no even maxima} \}| = h_n.\end{equation}  
 But this  is implicit in the work of Dumont and Randrianarivony  \cite{sur_une_extension} and stated explicitly in Corollary 10 (iii) of Randrianarivony \cite{Randrianarivony_Du_Fo_Poly}. 
\end{proof}

\subsection{The generating function formulas} \label{genfnsec}

In this subsection, we will use the bijections of the previous sections to derive generating function formulas  for the characteristic polynomial of the homogenized Linial arrangement. To do so, we use a formula (due independently to Randrianarivony \cite{Randrianarivony_Du_Fo_Poly} and Zeng \cite{Zeng_Du_Fo_Poly}) for the generating function of a multivariate polynomial that enumerates surjective staircases according to six statistics.

Let $f \in \mathcal E_{2n}$. Recall that a \emph{fixed point} of $f$ is a $j \in [2n-2]$ such that $f(j) = j$. A \emph{surfixed point} of $f$ is a $j \in [2n-2]$ such that $f(j) = j+1$. A \emph{maximum} of $f$ is a $j \in [2n-2]$  such that $f(j) =2n$.  Consider the following six statistics on $f$:
\begin{itemize}
\item $\text{fd}(f)$, the number of   fixed points $j$ of $f$ such that $f^{-1}(j) \setminus \{j\}$ is nonempty  (\emph{doubled fixed point})
\item $\text{fi}(f)$,   the number of  fixed points $j$ of $f$ such that $f^{-1}(j) = \{j\}$ (\emph{isolated fixed point})
\item $\text{sd}(f)$, the number of  surfixed points $j$ of $f$ such that $f^{-1}(j+1) \setminus \{j\}$ is nonempty (\emph{doubled surfixed point})
\item $\text{si}(f)$, the number of  surfixed points $j$ of $f$ such that $f^{-1}(j+1) = \{j\}$ (\emph{isolated surfixed point})
\item $\text{mo}(f)$,  the number of maxima $j$ of $f$ that are odd (\emph{odd maximum})
\item $\text{me}(f)$,  the number of maxima $j$ of $f$ that are even (\emph{even maximum}).
\end{itemize}

The \emph{generalized Dumont-Foata polynomial} $\Lambda_{2n}(x,y,z,\bar{x},\bar{y},\bar{z})$ (introduced by Dumont in \cite{Dumont_Foata}) is defined by
$$\Lambda_{2n}(x,y,z,\bar{x},\bar{y},\bar{z}) = \sum_{f \in \mathcal{E}_{2n}}x^{\text{mo}(f)}y^{\text{fd}(f)}z^{\text{si}(f)}\bar{x}^{\text{me}(f)}\bar{y}^{\text{fi}(f)}\bar{z}^{\text{sd}(f)}.
$$

\begin{theorem}[{Randrianarivony \cite[Theorem 4]{Randrianarivony_Du_Fo_Poly} and Zeng \cite[Theorem 5]{Zeng_Du_Fo_Poly}}]  \label{RZ theorem} 
$$\sum_{n\geq 1}\Lambda_{2n}u^n = \sum_{n\geq 1}\frac{(x+\bar{z})^{(n-1)}(y+\bar{x})^{(n-1)}u^n}{\prod_{k=0}^{n-1}(1- [(x+k)(\bar{y}-y) - (\bar{x}+k)(\bar{z}-z) - (x+k)(\bar{x}+k)] u},$$
where $a^{(n)} = a(a+1)\cdots(a+n-1)$.
\end{theorem}

\begin{lemma}  \label{cyclem} For all $n \ge 1$,
$$\sum_{\sigma \in \mathcal D_{2n}} t^{c(\sigma)} = \Lambda_{2n+2}(t,t,1,0,t,1),$$
where $c(\sigma)$ is the number of cycles of $\sigma$.
\end{lemma}

\begin{proof} Let $\phi$ be the bijection of Lemma~\ref{bijectionProp}.   By  Lemma~\ref{bijectionProp},
\begin{align*} \sum_{\sigma \in \mathcal D_{2n}} t^{c(\sigma)}  &= \sum_{\sigma \in \mathcal{D}_{2n}}  t^{\text{fi}(\phi(\sigma)) + \text{fd}(\phi(\sigma)) + \text{mo}(\phi(\sigma))}  
\\ &=   \sum_{\scriptsize \begin{array}{c} f \in \mathcal{E}_{2n+2} \\  \text{me}(f) = 0 \end{array} }  t^{\text{fi}(f) + \text{fd}(f) + \text{mo}(f)}  \\
& =  \Lambda_{2n+2}(t,t,1,0,t,1) .
\end{align*}

\end{proof} 

We  finally  have the tools needed to prove  (\ref{introgenchareq}), which is restated here.

\begin{theorem}\label{typeAGen} We have
\begin{equation} \label{typeAGenEq}\sum_{n\geq 1}\chi_{\Pi_{\Gamma_{2n}}}\! (t) \,u^{n} = \sum_{n\geq 1}\frac{(t-1)_{n}(t-1)_{n-1}u^n}{\prod_{k=1}^n(1-k(t-k)u)},\end{equation}
where $(a)_n = a(a-1)\cdots(a-(n-1))$.
\end{theorem}
 
 \begin{proof}  By Lemma~\ref{cyclem},  we have that 
 $$ \sum_{n \ge 1} \sum_{\sigma \in \mathcal D_{2n}} t^{c(\sigma)} u^n = \sum_{n\ge 2} \Lambda_{2n}(t,t,1,0,t,1) u^{n-1}  .$$
 Since $\Lambda_{2}(t,t,1,0,t,1)  = 1$,  it therefore follows from  Theorem ~\ref{RZ theorem} that
 \begin{align*} \sum_{n \ge 1} \sum_{\sigma \in \mathcal D_{2n}} t^{c(\sigma)} u^n &= -1+ \sum_{n\ge 1} \frac{ (t+1)^{(n-1)} t^{(n-1)} u^{n-1} } {\prod_{k=1}^{n-1}(1-(t+k)ku) }\\
 &= \sum_{n\ge 2} \frac{ (t+1)^{(n-1)} t^{(n-1)} u^{n-1} } {\prod_{k=1}^{n-1}(1-(t+k)ku) }\\
&=  \sum_{n\ge 1} \frac{(t+1)^{(n)} t^{(n)} u^{n}} { \prod_{k=1}^{n}(1-(t+k)ku)}.
 \end{align*}
 Now by Theorem~\ref{mobth}, we have
  \begin{align*}\sum_{n\geq 1}\chi_{\Pi_{\Gamma_{2n}}}\! (t) \,u^{n} &= - \sum_{n \ge 1} \sum_{\sigma \in \mathcal D_{2n}} (-t)^{c(\sigma)-1} u^n \\
&= t^{-1}  \sum_{n\ge 1} \frac{(-t+1)^{(n)} (-t)^{(n)} u^{n}} { \prod_{k=1}^{n}(1-(-t+k)ku)} \\
&=  \sum_{n\ge 1} \frac{(t-1)_{n} (t-1)_{n-1} u^{n}} { \prod_{k=1}^{n}(1+(t-k)ku)}.
 \end{align*}
 \end{proof}

\begin{remark} Theorem~\ref{chiDlem}  can also be used to prove Theorem~\ref{typeAGen}.  Indeed,
by Theorem~\ref{chiDlem} and Lemma~\ref{bijectionProp},
$$ \chi_{\Pi_{\Gamma_{2n}}}(t) = \Lambda_{2n}(1-t,1-t,1,0,-t,1) .
$$
Equation~(\ref{typeAGenEq}) now follows from the Randrianarivony-Zeng formula given in Theorem~\ref{RZ theorem}.
\end{remark}

It follows from (\ref{geneq}) and (\ref{heteq2}) that  equation~(\ref{typeAGenEq}) can  be viewed as a unifying generalization of two of the four generating function formulas stated in Section~\ref{gensec}.  
Indeed, when $t$ is set equal to $0$, equation~(\ref{typeAGenEq}) reduces to the Barsky-Dumont formula (\ref{BD2eq}) for the Genocchi numbers.  When $t$ is set equal to $1$, equation~(\ref{typeAGenEq}) reduces to the first Barsky-Dumont formula (\ref{BD1eq}) for the median Genocchi numbers.  We will see that the next result is  a unifying generalization of the other two generating function formulas stated in Section~\ref{gensec}.   

\begin{theorem} We have
 \begin{equation} \label{QuotCharTh} \sum_{n\geq 1}\chi_{\Pi_{\Gamma_{2n+2}}}\! (t) \,u^{n} = (t-1) \sum_{n\geq 1}\frac{\left((t-1)_{n}\right)^2 u^{n}}{\prod_{k=1}^{n}\left(1-k(t-k)u\right)}.\end{equation}  Equivalently,
\begin{equation} \label{QuotChar2Th} \sum_{n\geq 1}\chi_{L_{2n+2}}\! (t) \,u^{n} = \sum_{n\geq 1}\frac{\left((t-2)_{n-1}\right)^2 u^{n}}{\prod_{k=1}^{n}\left(1-k(t-k)u\right)},
\end{equation} 
where $L_{2n}$ is the geometric semilattice defined in Section~\ref{redsec}.
\end{theorem}

\begin{proof}
 By Theorem \ref{QuotCharLem} and Lemma~\ref{bijectionProp} we have that 
 for all   $n \ge 1$, $$\chi_{\Pi_{\Gamma_{2n+2}}}\! (t)= (t-1)^3 \Lambda_{2n}(1-t,2-t,1,0,1-t,1).$$
 The result now follows from the Randrianarivony-Zeng formula given in Theorem~\ref{RZ theorem}.

The equivalency of (\ref{QuotCharTh}) and (\ref{QuotChar2Th}) follows from (\ref{chiLeq}). \end{proof}

Equation (\ref{QuotChar2Th})  can   also  be viewed as a unifying generalization of two of the generating function formulas given in Section~\ref{gensec}.
Indeed, by Corollary~\ref{quotcor}, when $t$ is set equal to $0$, equation (\ref{QuotChar2Th})  reduces to the  Carlitz-Riordan-Stein formula (\ref{BD4eq}) for the    Genocchi numbers.
 When $t$ is set equal to $1$, equation (\ref{QuotChar2Th})  reduces to the second  Barsky-Dumont formula (\ref{BD3eq}) for the  the  median Genocchi numbers.  It is interesting that  the generating function formula (\ref{typeAGenEq}) for the characteristic polynomial of $ \Pi_{\Gamma_{2n}}$ reduces to one set of generating function formulas given in Section~\ref{gensec}, while the generating function formula (\ref{QuotChar2Th}) for the characteristic polynomial of $ L_{2n}$ reduces to the other set.

When we set $t=-1$ in (\ref{QuotCharTh}) we obtain a formula for the generating function of the median Genocchi numbers that is not listed in Section~\ref{gensec},
namely
$$\sum_{n\ge 0} h_n u^n = 1+ 2 \sum_{n\ge 1} \frac{ (n!)^2 u^n}{ \prod_{k=1}^{n-1} (1+k(k+1)u) }.$$
This formula might be known, but we have not seen it in the literature.
  
\begin{remark} Many of the  results in this section for $V=[2n]$ are extended to general Ferrers graphs $\Gamma_{V}$ in a forthcoming paper \cite{Other_Ferrers_Shapes}.
\end{remark} 
\section{Other  permutation models}  \label{othersec} 
In this section, we show that a formula of Chung and Graham \cite{cover_polynomial} for  chromatic polynomials of incomparablity graphs yields an interpretation  of the median Genocchi numbers in terms of yet another class of Dumont like permutations.  This  leads us to  conjecture an analogous interpretation of  the Genocchi numbers.  

A {\em drop} of a permutation $\sigma$ on a finite set of positive integers is a pair $(i, \sigma(i))$ for which $i > \sigma(i)$.  We say a drop $(i, \sigma(i))$ is an {\em even-odd  drop} if $i$ is even and $\sigma(i)$ is odd.  Let $d(n,k)$ be the number of permutations in $\sg_n$ with exactly $k$ drops that are {\em not} even-odd.
For example  the cycle $(1,3,2) \in \sg_3$ has two drops $(2,1)$ and $(3,2)$; the first drop is even-odd and  the second drop is not.    Hence this cycle counts towards $d(3,1)$.
One can check that $d(3,0) = 2$, $d(3,1) = 4$, and $d(3,2) = 0$.

\begin{theorem} \label{newcharth} For all $n \ge 0$,
\begin{equation} \label{newchareq} \chi_{\mathcal L (\mathcal H_{2n-1})}(t) = \frac{1}{(2n)!} \sum_{k=0}^{2n-1} d(2n,k) (t+1)^{(k)} (t-1)_{2n-1-k},\end{equation}
where $x^{(m)} = x(x+1) \cdots (x+m-1)$ and  $x_m := x(x-1)\cdots (x-m+1)$.
\end{theorem} 

Before proving this theorem we state a few consequences. By setting $t=-1$ in~(\ref{newchareq}) and using Hetyei's result (\ref{heteq2}), we obtain yet another combinatorial interpretation of the median Genocchi numbers.

\begin{corollary} \label{newMedCor} For all $n \ge 1$,
$$ h_n = |\{\sigma \in \sg_{2n} : \sigma \mbox{ has only even-odd drops}\}|.$$
\end{corollary}

By setting $t=0$ in~(\ref{newchareq}) and applying (\ref{geneq}) we obtain yet another formula for the Genocchi numbers.
\begin{corollary} \label{Newgencor} For all $n \ge 1$,
$$ g_n = \frac{1}{(2n)!} \sum_{k=0}^{2n-1} (-1)^{k} d(2n,k) k! (2n-1-k)! .$$
\end{corollary}

A much nicer combinatorial formula for the Genocchi numbers analogous to that of Corollary~\ref{newMedCor} is given  in the following conjecture.
\begin{conjecture} \label{newGenCon} For all $n \ge 1$, $g_n$ is equal to the number of cycles on $[2n]$ with only even-odd drops.
\end{conjecture} 
We have verified this conjecture by computer for $n \le 6$.
The following  conjecture along with Theorem~\ref{genth} implies both Corollary~\ref{newMedCor} and Conjecture~\ref{newGenCon}.
The first consequence follows by Theorem~\ref{mobth}.

\begin{conjecture}  \label{newGenCon2} Let $A$ be a finite subset of $Z_{> 0}$.  Then the  number of cycles on $A$ with only even-odd drops is equal to the number of D-cycles on $A$.  Consequently, for all $n \ge 1$,
\begin{enumerate} 
\item if 
 $\pi \in \Pi_{\Gamma_{2n}}$ then $|\mu_{ \Pi_{\Gamma_{2n}}}(\hat 0,x)|$ is equal to the number of permutations on $[2n]$ with only even-odd drops and   cycle support $\pi$.
 \item for all $j$, the number of permutations on $[2n]$ with $j$ cycles and with only even-odd drops is equal to the number of D-permutations with $j$ cycles.
 \end{enumerate}
 \end{conjecture}
 We have verified (2) by computer for $n \le 6$.
 
\begin{proof}[Proof of Theorem~\ref{newcharth}]

The \emph{incomparability graph} $\text{inc}(P)$ of a finite poset $P$ is the graph  with vertex set $P$ and edge set $$E=\{ \{x,y\} : x \mbox{ and  $y$ are incomparable in } P \}.$$
Given a poset $P$, we say that a permutation $\sigma $ of the vertices of $P$ has a \emph{$P$-drop} $(x,\sigma(x))$ if $x >_P \sigma(x)$, and write $d(P,k)$ for the number of permutations of $P$ with exactly $k$ $P$-drops.

Next we observe that $\Gamma_{2n}$ is the incomparability graph  of a poset.  
Indeed, $\Gamma_{2n}$  is the incomparability graph of the poset $P_{2n}$ on $[2n]$ with order relation given by $x \leq_{P_{2n}} y$ if:
\begin{itemize}
\item $x \leq y$ and $x$ and $y$ have the same parity
\item $x < y$, $x$ is even and $y$ is odd.
\end{itemize}  
The Hasse diagram of $P_6$ and its incomparability graph $\Gamma_6$ are given below. 

\vspace{.1in}\begin{center}
\includegraphics{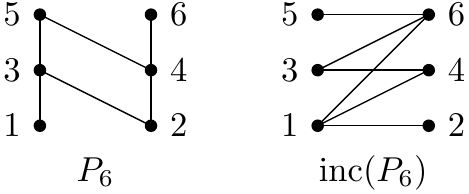}
\end{center}

Recall from Theorem~\ref{bondth} and Section~\ref{bondsubsec} that
$$\chi_{\mathcal L (\mathcal H_{2n-1})}(t)  = \chi_{\Pi_{\Gamma_{2n}}} (t) = t^{-1} {\rm ch}_{\Gamma_{2n}}(t),$$ 
where ${\rm ch}_G(t)$ denotes the chromatic polynomial of a graph $G$. 
Chung and Graham \cite[Corollary 7]{cover_polynomial} prove that for any finite poset $P$,
  the chromatic polynomial ${\rm ch}_{\text{inc}(P)}$ has the following expansion:
\begin{equation}\label{CGeq} {\rm ch}_{\text{inc}(P)}(t) = \sum_{k=0}^{|P|-1}d(P,k) \binom{t+k}{|P|}.\end{equation}
Hence since $\Gamma_{2n} =\text{inc}(P_n)$, 
$$\chi_{\mathcal L (\mathcal H_{2n-1})} (t) =\frac{1}{(2n)!} \sum_{k=0}^{2n-1} d(P_{2n},k) (t+1)^{(k)} (t-1)_{n-1-k}. $$
Now note that $(i,\sigma(i))$ is a $P_{2n}$-drop of $\sigma \in \sg_{2n}$ if and only if it is a drop of $\sigma$ that is not even-odd.  Hence $d(P_{2n},k)= d(2n,k)$ and the result holds. 
\end{proof}

We define an {\em even-odd descent} of a permutation $\sigma$ to be a descent   $\sigma(i)> \sigma(i+1)$ such that $\sigma(i)$ is even and $\sigma(i+1)$ is odd.  Since Foata's first fundamental transformation takes drops to descents,  the word ``drop''  can be  replaced by the word ``descent" in the definition of $d(2n,k)$ without changing the validity of  Theorem~\ref{newcharth} and Corollary~\ref{Newgencor}.  Similarly, Corollary~\ref{newMedCor} remains true if we replace the word ``drop" with ``descent".  

There is a  known permutation model for the Genocchi numbers that resembles the one we have been discussing.   Kitaev and  Remmel \cite{Classifying_Descents, Classifying_Descents_mod_k} conjectured and Burstein, Josuat-Verg\`{e}s, and Stromquist \cite{New_Dumont}  proved that the set of permutations in $\sg_{2n}$ with only even-even descents has cardinality equal to 
$g_{2n-2}$, where an even-even descent is defined in the obvious way.

\section{Further results: Type B and Dowling arrangements} \label{dowsec}
In a forthcoming paper \cite{LaWa}, we will present a type B analog of our results 
and a Dowling arrangement generalization that unifies our results in types A and B.
We define the {\em type B homogenized Linial arrangement} to be the hyperplane arrangement in $\R^{2n}$,
$$ \mathcal{H}^B_{2n-1} =\{ x_i \pm x_j = y_i : 1 \le i < j \le n \} \cup \{x_i =y_i: i = 1\dots, n\}.$$
Our type B analog of (\ref{introgenchareq})  yields the following generating function formula for the number of regions $r(\mathcal{H}^B_{2n-1}) $ of $\mathcal{H}^B_{2n-1}$:
  \begin{equation} \label{BBDeq} \sum_{n\geq 1} r(\mathcal{H}^B_{2n-1}) x^{n} = \sum_{n\geq 1}\frac{(2n)!x^n}{\prod_{k=1}^n( 1 + 2k(2k+1)x)}.\end{equation}

For $m \ge 1$, let $\omega_m$ be the primitive $m$th root of unity $e^{\frac{2\pi i }{m}}$.  For $m,n \ge 1$, we define the homogenized {\it Linial-Dowling arrangement}   to be the hyperplane arrangement in $\C^{2n}$ given by
 \begin{equation*} \mathcal H^m_{2n-1} :=  \{x_i -\omega_m^l x_j = y_i: 1 \le i < j \le n, \, 0 \le l < m\} \cup \{x_i =y_i : 1 \le i \le n \}.\end{equation*}  
 By intersecting $\mathcal{H}_{2n-1}^m$ with the subspace $y_1=y_2=\cdots = y_n = 0$, one gets the \emph{Dowling arrangement} $\mathcal A_n^m$ in $\C^n$. The intersection lattice of $\mathcal A_n^m$ is isomorphic to $\Pi_{n+1}$ when $m=1$ and to the type $B$ partition lattice $\Pi_n^B$ when $m=2$.
   By introducing this  Dowling analog of the homogenized Linial arrangement, we obtain unifying generalizations of the  types A and B results.  For instance, we obtain  the $m$-analog of  (\ref{introgenchareq}),
$$\sum_{n\geq 1}\chi_{ \mathcal L(\mathcal H^m_{2n-1})}(t) \, x^n =  \sum_{n\geq 1}\frac{  (t-1)_{n,m} (t-m)_{n-1,m} \,x^n}{\prod_{k=1}^n(1-mk(t-mk)x)},$$
where $(a)_{n,m} = a(a-m)(a-2m)\cdots(a-(n-1)m)$.
Note that this reduces to 
(\ref{BBDeq}) when we set $m=2$ and $t=-1$.

 We also obtain an $m$-analog of the  formula $g_{n}=|\mu_{\mathcal L(\mathcal{H}_{2n-1})}(\hat 0, \hat 1)| $ involving a  well studied polynomial analog of the Genocchi numbers known as the Gandhi polynomials \cite{Gandhi_article, sur_une_extension}.
 
 For a synopsis of some of our results on the homogenized Linial-Dowling arrangement, see the extended abstract \cite[Section 4]{Extended_Abstract}.

\section*{Acknowledgements}
The authors thank Jos\'{e} Samper for a valuable suggestion pertaining to bipartite graphs.    MW thanks G\'abor Hetyei for introducing her to his work on this topic when she visited him at the University of North Carolina, Charlotte during her Hurricane Irma evacuation, and  for his hospitality.

\bibliographystyle{amsplain}
\bibliography{HLBiblio}

\end{document}